\documentclass[12pt]{amsart}
\usepackage{amssymb,amsthm,amsmath,amstext}
\usepackage{mathrsfs}  
\usepackage{bm}        
\usepackage{mathtools} 

\usepackage[left=1.28in,top=.8in,right=1.28in,bottom=.8in]{geometry}

\usepackage{graphicx}
\usepackage[all]{xy}

\theoremstyle{plain}
\newtheorem{theorem}{Theorem}[section]
\newtheorem{prop}[theorem]{Proposition}
\newtheorem{lemma}[theorem]{Lemma}

\newtheorem{cor}[theorem]{Corollary}

\newtheorem*{conj1}{Conjecture 1} 
\newtheorem*{conj2}{Conjecture 2} 

\theoremstyle{definition}

\theoremstyle{remark}
\newtheorem{remark}[theorem]{Remark}

\newcommand{\sheaf}[1]{\mathscr{#1}}

\newcommand{\UU}{\sheaf{U}}

\newcommand{\BB}{\sheaf{B}}
\newcommand{\PP}{\sheaf{P}}
\newcommand{\XX}{\sheaf{X}}
\newcommand{\YY}{\sheaf{Y}}

\newcommand{\QQ}{\sheaf{Q}}

\newcommand{\pp}{\mathfrak{p}}

\newcommand{\Group}[1]{\mathbf{#1}}

\newcommand{\Z}{\mathbb Z}

\newcommand{\Q}{\mathbb Q}
\newcommand{\F}{\mathbb F}

\newcommand{\Gm}{\Group{G}_{\text{m}}}

\usepackage{hyperref}
\usepackage{color}

  \usepackage[OT2, T1]{fontenc}
\DeclareSymbolFont{cyrletters}{OT2}{wncyr}{m}{n}
\DeclareMathSymbol{\Sha}{\mathalpha}{cyrletters}{"58}

\hypersetup{colorlinks=true,linkcolor=blue,anchorcolor=blue,citecolor=blue,letterpaper=true}

\begin{document}

\title[ Local-global principle]
{ Local-global principle for unitary groups   over function fields of
  $p$-adic curves}

\author[Parimala]{R.\ Parimala }
\address{Department of Mathematics \& Computer Science \\ %
Emory University \\ %
400 Dowman Drive~NE \\ %
Atlanta, GA 30322, USA}
\email{praman@emory.edu}

\author[Suresh]{V.\ Suresh}
\address{Department of Mathematics \& Computer Science \\ %
Emory University \\ %
400 Dowman Drive~NE \\ %
Atlanta, GA 30322, USA}
\email{suresh.venapally@emory.edu}

\date{}

\begin{abstract}  Let $K$ be a $p$-adic field and $F$ the function field of a curve over $K$.
Let $G$ be a connected linear algebraic group over $F$ of classical type. Suppose the prime $p$ is a good prime for $G$.
Then we prove that projective homogeneous spaces under $G$ over $F$ satisfy a local global principle for rational points
with respect to discrete valuations of $F$. If $G$ is a semisimple simply connected group over $F$, then we also prove that 
 principal  homogeneous spaces under $G$ over $F$ satisfy a local global principle for rational points
with respect to discrete valuations of $F$. 
\end{abstract} 
 
\maketitle

\def\ZZ{${\mathbf Z}$}
\def\ih{${\mathbf H}$}
\def\RR{${\mathbf R}$}
\def\IF{${\mathbf F}$}
\def\QQ{${\mathbf Q}$}
\def\IP{${\mathbf P}$}

\date{}

 Let $k$ be a number field and $G$ a semisimple simply connected linear algebraic group over $k$.
 Classical Hasse principle asserts that a principal homogeneous space under $G$ over $k$ has a rational point if it has rational points  
 over all completions of $k$.  This is a theorem due to Kneser (classical groups), Harder (for exceptional groups other than $E_8$)
 and Chernousov (for $E_8$). Harder also proves a Hasse principle for rational points on projective homogeneous spaces  under 
 connected linear algebraic groups over $k$. 
 
 Questions related to Hasse principle have been extensively studied over `semiglobal fields', namely function fields of curves over 
 complete discretely valued fields with respect to their discrete valuations.  Considerable  progress is has been made  possible due 
 to the patching techniques  of Harbater, Hartmann and Krashen. One could look for analogous Hasse principles 
 for simply connected groups in this context.  However Hasse principle fails for simply connected groups
  in this generality (\cite{CTHHKPS}).  If $K$ is a $p$-adic  field and $G$ is  semisimple simply connected quasi split linear algebraic 
  group over the function field of  a curve over $K$ with $p \neq 2, 3, 5$, it was proved in (\cite{CTPS}) that Hasse principle 
  holds for $G$. This led to the following  two conjectures (\cite{CTPS}).
  
  {\it Let $F$ be the function field of a $p$-adic curve and $\Omega_F$  the set of all discrete valuations of $F$.
  For $\nu \in F$, let $F_\nu$ be the completion of $F$ at $\nu$. }
  
  \begin{conj1} 
  \label{proj} Let $Y$ be a projective homogeneous space under a connected linear algebraic group $G$ over $F$. 
  Then $Y$ satisfies Hasse principle with respect to $\Omega_F$.
  \end{conj1}
  
 \begin{conj2}
  \label{prin}  Let $G$ be a semisimple simply connected linear algebraic group over $F$ and $Y$ a principal homogeneous space
  under $G$ over $F$. Then $Y$ satisfies Hasse principle with respect to $\Omega_F$. 
  \end{conj2}

  There has been considerable  progress towards these conjectures for classical groups in the `good characteristic case'.
  Let $G$ be a semisimple simply connected linear algebraic group  of classical type over $F$. 
  We say that the prime $p$ is {\it good} for $G$,  if  $p \neq 2$ for $G$  of type 
  $B_n$, $C_n$, $D_n$ ($D_4$ non trialitarian) and $p $ does not divide $n+1$ for $G$ is of type $^1\!\!A_n$ and 
  $p $ does not divide $2(n+1)$ for  $G$ is of type $^2\!\!A_n$.   Let $G$ be any connected linear algebraic group over $F$.
  We say that $G$ is of  {\it classical type}  if every factor of the simply connected cover $\tilde{G}$ 
   of the  semi-simplification  of $G/Rad(G)$
  is of classical type. 
  We say that $p$ is {\it good} for $G$ if $p$ is good for every factor of $\tilde{G}$.

Suppose $p \neq 2$.  It was proved in (\cite{CTPS}) that a quadratic form  $q$ over $F$ of rank at least 3 is isotropic 
  over $F$ if and only if $q$ is  isotropic over $F_\nu$ for all $\nu \in \Omega_F$. 
 A local global principle  for generalized Severi-Brauer varieties, under an assumption on the  roots of unity in $F$, 
 is due to Reddy and Suresh (\cite{RS}).
 Let $A$ be a central simple algebra over $F$ with an involution $\sigma$ of either kind. 
   If $\sigma$ is of the second kind, then assume that 
 ind$(A) \leq 2$. Let $h$ be an hermitian form over $(A, \sigma)$. Then 
 Wu (\cite{Wu}) proved the validity of conjecture 1 for the  unitary groups of  $(A, \sigma)$. Hence  conjecture 
 1  holds for  all groups of type $B_n$,  $C_n$, $D_n$ and for special groups of type  $^1\!\!A_n$ and $^2\!\!A_n$ in the good characteristic case (\cite[Corollary 1.4]{Wu}).

  Conjecture 2  for groups of type $B_n$, $C_n$, $D_n$ is due to  Hu and Preeti independently (\cite{Hu}, \cite{preeti}). 
  Conjecture 2 for $G = SL_1(A)$ with index of $A$ square free is a consequence of the injectivity 
  of the Rost invariant due to Merkurjev-Suslin (\cite{MS})  and a result of Kato (\cite{Ka})  on the injectivity of $H^3(F, \Q/\Z(2)) \to \prod_{\nu \in 
  \Omega_F}H^3(F_\nu, \Q/\Z(2))$.  The case $^2\!\!A_n$, namely the unitary groups of algebras of index at most 2 with unitary 
  involution is due to Hu and Preeti (\cite{Hu}, \cite{preeti}). 
  
  The two main open cases concerning Conjectures 1 and 2  for classical groups  were types $^1\!\!A_n$ and $^2\!\!A_n$.
  The  Conjecture 2  for  $^1\!\!A_n$, namely  a local global principle for reduced norms in the good characteristic case was settled by 
  the authors and Preeti (\cite{PPS}). 
  
  The aim of this paper is to settle Conjecture 1 and Conjecture 2
  in the affirmative in the good characteristic case for all  groups of types $^1\!\!A_n$ and  $^2\!\!A_n$, 
   thereby completing the proof for all   classical groups  in the good characteristic case.  
     In fact we prove the following.

\begin{theorem}(cf. \ref{proj-pgl}) 
Let $K$ be a $p$-adic   field   
  and $F$ the function field of a smooth projective  curve over $K$.   Let $A$ be a central simple algebra over $F$ of index 
  coprime to $p$.  Then Conjecture 1 holds for $PGL(A)$.   
 \end{theorem}

  \begin{theorem}(cf. \ref{isotropic1}) Let $K$ be a $p$-adic  field and
 $F_0$ the  function field of  a  smooth projective  curve over $K$.   
Let $F/F_0$ be a quadratic extension and $A$ be  central simple  algebra over $F$ of index $n$
 with an $F/F_0$- involution $\sigma$. 
 Suppose that $2n$ is coprime to  $p$. 
  Let $h$ be an hermitian form over $(A, \sigma)$. 
  If $A = F$, then assume that the  rank of $h$ is  at least 2. 
  Then Conjecture 1 holds for $ U(A, \sigma, h)$.   
  \end{theorem}

   \begin{theorem}(cf. \ref{lgp-unitary})  Let $K$ be a $p$-adic  field and
 $F_0$ the  function field of  a  smooth projective  curve over $K$. 
Let $F/F_0$ be a quadratic extension and $A$ a central simple algebra over $F$ of index $n$
 with an $F/F_0$- involution $\sigma$. 
 Suppose that $2n$ is coprime to $p$.  
  Then Conjecture 2  holds for $SU(A, \sigma)$.   
   \end{theorem}

  As a consequence  we   have the following.

  \begin{theorem}(cf. \ref{conjecture1}) 
   Let $K$ be a $p$-adic  field and
 $F$ the  function field of  a  smooth projective  curve over $K$. 
  Let $G$ be a connected linear algebraic group over $F$ of classical type ($D_4$ nontrialitarian) 
  with $p$ good for $G$.   Then Conjecture 1 holds for $G$.
   \end{theorem}   
  
  \begin{theorem} 
  (cf. \ref{conjecture2}) Let $K$ be a $p$-adic  field and
 $F_0$ the  function field of  a  smooth projective  curve over $K$. 
  Let $G$ be a  semisimple simply connected  linear algebraic group over $F$   with $p$ good for $G$. 
  If $G$ is of  classical type ($D_4$ nontrialitarian), then Conjecture 2 holds for $G$.
   \end{theorem}

  Here is an outline of the structure of the paper.  The plan is to reduce the questions on local global principle
   with respect to discrete valuations to one for the patching fields  in the setting of Harbater, Hartmann and Krashen (\cite{HHK1}) and then to deal with the question 
   in the patching setting.  The reduction to the patching setting requires an understanding of the structure of central 
   simple algebras with involutions of the second kind over the branch fields (\cite{HHK1}), which are 2-local  fields. 
   This leads to the study of cyclic extensions over quadratic extensions  of local fields with zero corestriction.
  Let $F_0$ be a field,  $F/F_0$ be a quadratic extension and  $L/F$  a cyclic extension of degree coprime 
  to char$(F_0)$.      It was proved in (\cite[Proposition 24]{HKRT}) that   
  the corestriction of $L/F$ from $F$ to $F_0$ is zero if and only if $L/F_0$ is a dihedral extension.  
In \S \ref{cores}  we reprove this statement for the sake of completeness and deduce some consequences
for dihedral extensions. 
 In \S \ref{dihedral}  we study dihedral extensions over an arbitrary fields. 
  In \S \ref{dihedral-local} we describe all   dihedral extensions over local fields. 
In \S \ref{alg-2-local} and \S \ref{alg-2dimlocal} we describe the structure of central simple algebras with unitary 
involutions over 2-local fields and 2-dimensional complete fields with finite residue fields. These fields surface 
in the patching setting.
In \S\ref{whitehead}, we show that the reduced Whitehead groups over 2-local fields are trivial. 
In \S\ref{lgp-gl}, we prove a local global principle for  generalized Severi-Brauer varieties without any assumption on the existence  of roots of unity, completing the proof of Conjecture 1 for groups of type $^1\!\!A_n$.
 In \S  \ref{lgp-herm}, we prove a local global principle for isotropy of hermitian forms over division algebras 
 with unitary involutions. The idea is to construct good maximal orders  invariant under involution over 
 2-dimensional complete regular local rings. This is possible due to the complete understanding of 
 the structure of the algebras with unitary involutions studied in \S \ref{alg-2-local}. 
 This settles Conjecture 1 for groups of type $^2\!\!A_n$ in the  good characteristic case. 
In \S \ref{lgp-patching},  we prove the local global  principle for principal homogeneous spaces under  simply 
connected unitary groups in the patching setting. Finally, in 
\S \ref{lgp-dvr} we prove the local global principle for  special unitary groups with respect to discrete valuations, thereby completing the validity of Conjecture 2 for groups of type  $^2\!\!A_n$.
   More generally (cf. \ref{conjectures}), we  prove  Conjectures 1 and 2 for groups of  classical type   over function fields of curves over local fields.
 
  Throughout  this paper,  a projective homogeneous space $Z$  under a connected  linear algebraic group $G$  is a projective variety $Z$ with transitive $G$-action  over the separable closure such that the stabilizer  is a parabolic subgroup.  
   
 \section{Preliminaries}

 \begin{lemma}
 \label{finitefields} Let $\F_q$ be the finite field  with $q$ elements and $\F_{q^2}$ the degree two extension of $\F_q$.
 Suppose  $q$ is odd and $\sqrt{-1} \not\in \F_q$. Then $\F_{q^2} = \F_q(\sqrt{-1})$.
 Let $d$ be the maximum integer such that $\F_{q^2}$ contains a primitive $2^d$-th root of unity $\rho$.
 Then $N_{\F_{q^2}/\F_q} (\rho) = -1$.  
 \end{lemma}

 \begin{proof} Since $\sqrt{-1} \not\in \F_q$ and $q$ is odd, we have 
 $\F_q^*/\F_q^{*2} = \{ 1, -1\}$.  Since there is a unique extension of degree 2 of $\F_q$, 
 $\F_{q^2} = \F_q(\sqrt{-1})$. Let $d$ be the maximum integer such that $\F_{q^2}$ 
 contains a primitive $2^d$-th root of unity $\rho$. Since  there is no $2^{d+1}$-th primitive root of unity in 
 $\F_{q^2}$, $\rho\not\in \F_{q^2}^{*2}$. Hence $ \F_{q^2}^*/ \F_{q^2}^{*2} = \{ 1, \rho \}$.
 Since $N_{\F_{q^2}/\F_q} : \F_{q^2}^* \to \F_q^*$ is surjective, 
 $N_{\F_{q^2}/\F_q} : \F_{q^2}^*/ \F_{q^2}^{*2} \to \F_q^*/\F_{q}^{*2}$ is surjective.
 Hence  $N_{\F_{q^2}/\F_q} (\rho) = -1$.  
 \end{proof}

 \begin{cor}
 \label{localfields} Let $K_0$ be a  local field and $K/K_0$ the quadratic unramified extension. 
Suppose that  the  characteristic of the residue field  of $K_0$ is odd and $\sqrt{-1} \not\in K_0$. Then $K  = K_0(\sqrt{-1})$.
 Let $d$ be the maximum integer such that $K$ contains a primitive $2^d$-th root of unity $\rho$.
 Then $N_{K/K_0} (\rho) = -1$.  
 \end{cor}

 \begin{lemma}
 \label{finitefields2} Let $\F_q$ be the finite field  with $q$ elements and $\F_{q^2}$ the degree two extension of $\F_q$.
 Let $m \geq 1$.   Suppose  $q$ is odd and $\sqrt{-1} \not\in \F_q$.  
 If   $\F_{q^2}$ contains a primitive $2^{m+1}$-th root of unity,  then $\F_q^* \subset \F_{q^2}^{*2^m}$. 
 \end{lemma}
 
 \begin{proof} Since $\sqrt{-1} \not\in \F_q^*$, the only $2^m$-th roots of unity in $\F_q$ are $\pm 1$.
 Hence we have an exact sequence of groups  $1 \to \{ \pm 1 \}\to  \F_q^* \to \F_q^{*2^n} \to 1$, where the
 last map is given by $x \to x^{2^n}$.  Thus  the order of $\F_q^*/\F_q^{*2^m}$ is 2.
 Since $-1 \not\in \F_q^{*2}$, $-1 \not\in \F_q^{*2^m}$ and $\F_q^* = \F_q^{*2^m} \cup (-1) \F_q^{*2^m}$.
 Since $\F_{q^2}^*$ contains a primitive $2^{m+1}$-th root of unity,  $-1 \in \F_{q^2}^{*2^m}$.
 Thus $\F_q^* \subset \F_{q^2}^{*2^m}$. 
 \end{proof}
 
 \begin{cor}
 \label{localfields2} Let $K_0$ be a  local field and $K/K_0$ the quadratic unramified extension. 
Suppose that  the  characteristic of the residue field of $K_0$  is odd and $\sqrt{-1} \not\in K_0$.
 Let $m \geq 1$.     If   $K$ contains a primitive $2^{m+1}$-th root of unity,  then every unit in the valuation  ring of  $K_0$
 is in $K^{*2^m}$. 
 \end{cor}
 
 \begin{lemma}
 \label{finitefields3} 
 Let $\F_q$ be the  finite field  with $q$ elements. 
 Let $m \geq 1$  be   coprime to  $q$.  
 Suppose that $\F_q$ does not contain any nontrivial   $m^{\rm th}$ root of unity.
 Then $\F_q^* = \F_q^{*m}$.  
  \end{lemma}
 
 \begin{proof} Since  $ \F_q^*$ does not contain nontrivial  $m^{\rm th}$ root of unity,  
 the only $m^{\rm th}$  root  of unity in $\F_q$ is $ 1$.
 Hence  the homomorphism $ \F_q^* \to \F_q^{*^m} $,   $x \mapsto x^{^m}$,  is an isomorphism. 
 Thus $\F_q^*  =  \F_{q}^{*^m}$. 
 \end{proof}

 \begin{cor}
 \label{localfields3} 
 Let $K_0$ be a  local field. 
 Let $m \geq 1$ be   coprime  to  the  characteristic of the residue field of $K_0$. 
 Suppose that $K_0$ does not contain any nontrivial   $m^{\rm th}$ root of unity.
 Then  every unit in the valuation ring  of $K_0$ is  in $ K_0^{m}$.  
  \end{cor}

Let $F$ be a  discretely valued field  with  the valuation ring $R$ and   residue field $K$.
    We say that an element $a \in F$ is a {\it unit in } $F$ if $a \in R$ is a unit. 
 Let $n\geq 1$ be an integer coprime to char$(K)$.
Then we have the residue map $\partial : H^d(F, \mu_n^{\otimes i}) \to  H^{d-1}(F, \mu_n^{i-1})$. 
Let $H^d_{nr}(F, \mu_n^{\otimes i})$ be the kernel of $\partial$.
An element $\alpha \in H^d_{nr}(F, \mu_n^{\otimes i})$ is called an  {\it unramified} element. 
If  $F$ is complete, then we have an isomorphism $H^d(K, \mu_n^{\otimes i}) \simeq H^d_{nr}(F, \mu_n^{\otimes i})$. 

We end this section with the following  result on reduced norms.  

\begin{prop}
\label{reduced-norms}
Let $K$ be a  global  field   with no real orderings and  $F$  a complete discretely  valued field with residue field $K$. 
Let $A$ be  a central simple algebra over $F$ of index $n$  coprime to  char$(K)$.
Let $ (L, \sigma) \in H^1(K, \Z/n\Z)$ be the residue of $A$.  Let $\theta \in F^*$ be a unit.
If the image of $\theta  \in K^*$ is a norm from the extension $L/K$, then $\theta$ is a reduced norm from 
$A$. 
\end{prop}
 
\begin{proof} Let $E/F$ be the unramified extension with residue field $L$ and $\tilde{\sigma}$ 
a generator of Gal$(E/F)$  lifting   $\sigma$.   Let $R$ be the valuation ring of $F$ and  $\pi \in R$ be a parameter.
Then $A = A_0 + (E, \tilde{\sigma}, \pi)$  for some 
 central simple algebra $A_0$ over $F$ representing a class in $H^2_{nr}(F, \mu_n)$ (cf. \cite[Lemma 4.1]{PPS}).
 Since $F$ is complete and the image of $\theta$ in $K$ is a norm from $L/K$, 
 $\theta$ is a norm from $E/F$.  Hence $(E, \tilde{\sigma}, \pi) \cdot (\theta) = 0 \in H^3(F, \mu_n^{\otimes 2})$.
 Since $A_0$ is unramified on $R$ and  $\theta$ is a unit, $A_0 \cdot (\theta) \in H^3_{nr}(F,\mu_n^{\otimes 2})$.
 Since $K$ is a global field with no real orderings,   cd$(K) = 2$ and $H^3(K, \mu_n^{\otimes 2}) = 0$.
Hence $H^3_{nr}(F, \mu_n^{\otimes 2})= 0 $ and  $A_0 \cdot (\theta) = 0$.
In particular   $A\cdot (\theta) = 0 \in H^3(F, \mu_n^{\otimes 2})$ and by (\cite[Theorem 4.12]{PPS}) 
 $\theta$ is a reduced norm from $A$.  
\end{proof}

\section{Dihedral Extensions}
\label{dihedral}
  Let $G$ be a dihedral group of order  $2m \geq 4$.  Let  $\sigma$ and $\tau$ be the  generators of $G$   with 
  $\sigma^m = 1$, $\tau^2 = 1$ and $\tau \sigma \tau = \sigma^{-1}$.
  The subgroup  generated by $\sigma$ is  the {\it rotation}  subgroup of 
  $G$ and  for $0 \leq i \leq m-1$,  $\sigma^i \tau$ are the  {\it reflections}. 
  
  Let $F_0$ be a field and $E/F_0$  a field  extension. We say that $E/F_0$ is a {\it dihedral extension} if 
  $E/F_0$ is  Galois with Galois group isomorphic to a dihedral group.  In this section we prove some basic facts about dihedral extensions. 
  
   \begin{lemma}
 \label{subextndihedral}
 Let  $F_0$ be  a  field   and $E/F_0$ a  dihedral extension. Let $F$ be the fixed of the rotation subgroup of Gal$(E/F_0)$. 
 If $M/F$ is a  sub extension of $E/F$ with $M \neq F$, then $M/F_0$ is a dihedral extension. 
 \end{lemma}
 
 \begin{proof}    Let Gal$(E/F_0)$ be generated by $\sigma$ and $\tau$ with $\sigma^m = 1$, $\tau^2 = 1$ and 
 $\tau \sigma \tau = \sigma^{-1}$. Then $E/F$ is cyclic  with Gal$(E/F)$ generated by $\sigma$.
 Let $M/F$ is a sub extension of $E/F$. The extension  $M/F$ is cyclic with Gal$(M/F)$ generated by the restriction of $\sigma$ to 
 $M$. Since  $M = E^{\sigma^i}$ for some $i$ and $\tau \sigma^i \tau = \sigma^{-i}$,   the extension 
 $M/F_0$ is Galois with the Gal$(M/F_0)$ generated by the restriction of $\sigma$ and $\tau$ to $M$.
 Since $M \neq F$, the restriction of $\sigma$ to $M$ is nontrivial. Since $F \subset M$, the restriction of $\tau$ to $M$ is nontrivial.
 Hence $M/F_0$ is dihedral. 
 \end{proof}

  \begin{lemma}
  \label{subextndihedral-notgalois} Let $E/F_0$ be a dihedral extension and $F$ the fixed field of the rotation subgroup of 
  Gal$(E/F_0)$. Let $F_0\subseteq L  \subseteq E$ with   $F \not\subset L$. If   $L/F_0$  is Galois, 
   then   $[L : F_0] \leq 2$. 
  \end{lemma}
  
  \begin{proof} Suppose that  $F \not\subset L$ and $L/F_0$ is Galois. Let $M = FL$.
  Suppose that $L \neq F_0$.  Then $M \neq F$ and hence  $M/F_0$  dihedral (\ref{subextndihedral}). 
   Since $F \not\subset L$, $[M : F] = [L : F_0]$. 
  Since $L/F_0$  and $F/F_0$ are Galois   extensions, $M/F_0$ is Galois with Gal$(M/F_0)$ isomorphic to Gal$(L/F_0) \times $Gal$(F/F_0)$.
  Since the only  dihedral group which is isomorphic to a direct product of two nontrivial subgroups is $\Z/2 \times \Z/2$,   $[L : F_0]  = 2$.   
  \end{proof}

\begin{lemma}
\label{cores1}
Let $F_0$ be a field and $E/F_0$ a dihedral extension of degree $2m$ and $F$ the fixed field of the rotation subgroup of Gal$(E/F_0)$.
 Then there exist exactly $m$   subfields  $E'$ of $E$ containing $F_0$ with 
$[E' : F_0] = [E : F]$ and $E'F = E$. 
Further if $E'$ is  any such  subfield  of $E$ and $\ell_1$, $\ell_2$, $\cdots$, $\ell_r$ is any sequence of prime numbers 
with $[E : F] = \ell_1 \cdots \ell_r$,    then 
there   exist subfields 
  $F_0 = L_0 \subset L_1 \subset \cdots \subset L_r = E'$ with $[L_i : L_{i-1}] = \ell_i$.  
 \end{lemma}

\begin{proof}  Let $\sigma$ be a generator of the rotation subgroup of Gal$(E/F_0)$ and 
$\tau$ a reflection.    For  $0 \leq i\leq m-1$ $i$,  
let  $E_i = E^{\tau\sigma^i}$ the subfield of $E$ fixed by $\tau\sigma^i$. 
Then $[E : E_i] = 2$,  $[E_i : F_0] = m$ and $E_i F = E$. 
Since only elements of  order 2 in Gal$(E/F_0)$ which are not identity on $F$ are 
the reflections $\tau\sigma^i$, $0 \leq i \leq m-1$, any  $E'$ with the given properties 
coincides with $E_i$ for some $i$.

Let  $E' = E_i$ for some $i$.  
Suppose $m = \ell_1 \cdots \ell_r$ with $\ell_j$'s primes. 
Since $E/F$ is a cyclic extension, there exist subfields $F = M_0 \subset M_1  \subset \cdots \subset M_r = E$
such that $[M_j : M_{j-1} ] = \ell_i$ for all $i$. 
Then $L_j = E' \cap M_j$ have the required property.  
 \end{proof}

\begin{lemma}
\label{dihedral-norm1}  Let $F_0$ be a field   and $F/F_0$ a quadratic Galois extension. 
Let $m  \geq 2$  be coprime to char$(F_0)$.  Suppose that $F$ contains a primitive $m^{\rm th}$ root of unity $\rho$.
Let $a \in F_0^*$.  Suppose that  $[F(\sqrt[m]{a}) : F] = m$. 
 Then $ F(\sqrt[m]{a})/F_0$ is dihedral   if and only if   $N_{F/F_0}(\rho) = 1$.   
\end{lemma}

\begin{proof}  Let $E  = F(\sqrt[m]{a})$ and  $E' = F_0(\sqrt[m]{a})$.    Since $a \in F_0^*$, we have $E = E'F$.
Since $[E: F_0] = 2m$, we have  $[E : E'] = 2$. Let $\sigma$ be the automorphism of $F(\sqrt[m]{a})$ given by $\sigma(\sqrt[m]{a}) = \rho  \sqrt[m]{a}$ and
 $\tau $   the nontrivial automorphism of $E/E'$.
Since $\tau$ is nontrivial on $F$, it follows that $\tau \neq \sigma^i $ for any $i$. 
Hence $E/F_0$ is Galois and Gal$(E/F_0)$ is generated by $\sigma$ and $\tau$. 
 Since  the order of $\sigma$ is $m$ and $\tau^2 = 1$, Gal$(E/F_0)$ is dihedral if and only if $\tau \sigma \tau = \sigma^{-1}$.

We have $\tau \sigma \tau (\sqrt[m]{a} ) = \tau \sigma(\sqrt[m]{a}) = 
\tau (\rho \sqrt[m]{a}) = \tau(\rho) \sqrt[m]{a} =   \tau(\rho) \rho \rho^{-1}   \sqrt[m]{a}  = 
 \tau(\rho) \rho  \sigma^{-1} (\sqrt[m]{a})$. 
 Hence $\tau\sigma \tau = \sigma^{-1}$ if and only if $ N_{F/F_0}(\rho )  = \tau(\rho ) \rho  = 1$. 
 \end{proof}

We end this section with the following:

 \begin{lemma} 
 \label{ei-split} Let $F_0$ be a  field   and   $n \geq 2$  an integer  with $2n$ coprime to  char$(F_0)$.   
 Let $E/F_0$ be a dihedral extension of degree $2n$ and $\sigma$ and $\tau$ generators on Gal$(E/F_0)$
 with $\sigma$ a rotation and $\tau$ a reflection. 
 Let  $F = E^\sigma$ and  $E_i = E^{\sigma^i \tau}$ for $1 \leq i \leq n$. 
 Let $M/F_0$ be a field extension. 
 Suppose $F \otimes_{F_0} M $ is a field and $E\otimes _{F_0}  M$ is isomorphic to 
  $\prod_1^n (F \otimes_{F_0} M)$.
  Then there exists $i$   such that $E_i \otimes_{F_0} M \simeq  M  \times E'_i$ for  some $M$-algebra 
  $E'_i$. 
 \end{lemma}
 
\begin{proof}  The proof is  by induction on $n$. 
Suppose that $n = 2$. Then $F = F_0(\sqrt{a})$, $E_1 = F_0(\sqrt{b})$,  $E_2 = F_0(\sqrt{ab})$  and 
$E = F(\sqrt{b})$.  Suppose that $M(\sqrt{a})  = F \otimes_{F_0}M$ is a field and 
$E\otimes_{F_0}M$ is not a field. Then  $a$ is not a square in $M$ and  $E \otimes _{F_0}M \simeq  M(\sqrt{a})  \times M(\sqrt{a})$.
Then  either $b$ is a square in $M$ or $ab$ is a square in $M$.
Thus either $E_1 \otimes_{F_0}M  \simeq M  \times M$ or 
$E_2 \otimes_{F_0} \simeq M  \times M$. 

Suppose $n \geq 3$. Suppose that $M(\sqrt{a}) = F \otimes_{F_0}M$ is a field 
and $E \otimes_{F_0} M  \simeq   \prod_1^n M(\sqrt{a})$. 

Suppose $n$ is odd.  
Since $E_i \otimes_{F_0} F \simeq E$ and  $F/F_0$ is of degree 2,
it follows that $E_i \otimes_{F_0}M  \simeq \prod_1^r M  \times \prod_1^s M(\sqrt{a})$.  
Since $[E_i : F_0] = n$   is odd, $ r \geq 1$.

Suppose that $n$ is even. Then, by (\ref{cores1}), there exists  a quadratic extension $F_1/F_0$ contained in 
$E$ and $F_1 \neq F$.  Let $F' = FF_1$. Then $F'/F_0$ is a biquadratic extension.
Hence there is a quadratic extension $F_2/F_0$ contained in $F'$ with $F \neq F_2$ and $F_1 \neq F_2$.
Further every quadratic extension of $F_0$ contained in $E$ is either $F$ or $F_1$ or $F_2$. 
 Since every $E_i$ contains a quadratic extension of $F_0$ (\ref{cores1}) and $F \not\subset E_i$, 
  half of $E_i$ contain $F_1$ and the remaining half of $E_i$ contain $F_2$. 
  Further $E/F_1$  and $E/F_2$ are dihedral extensions of  degree $n$. 

Since $E \otimes_{F_0}M  \simeq   \prod_1^n  M(\sqrt{a})$, 
$F' \otimes_{F_0}M \simeq    M(\sqrt{a}) \times M(\sqrt{a})$. Thus, by the case  $n = 2$, either $F_1 \otimes_{F_0} M 
\simeq M \times M$ or 
$F_2 \otimes_{F_0}M \simeq M  \times M $.  Without loss of generality, assume that 
$F_1 \otimes_{F_0} M 
\simeq M \times M$.   Then $F_1$ is isomorphic to a subfield of $M$ and hence $M/F_1$ is an extension of fields.

Since   $F' = F_1(\sqrt{a})$ and $a$ is not a square in $M$,  $F' \otimes_{F_1}M$ is a field. 
Since  $E\otimes_{F_0} M \simeq   E\otimes_{F_1}F_1 \otimes_{F_0}M  \simeq E\otimes_{F_1}(M \times M) 
\simeq E\otimes_{F_1}M \times E\otimes_{F_1}M$ and 
$E\otimes_{F_0}M \simeq \prod_1^n M(\sqrt{a})$,  it follows that 
  $E\otimes_{F_1} M  \simeq 
 \prod_1^{n/2} M(\sqrt{a})$. 
 Since $E/F_1$ is dihedral and $[E : F_1] < [E : F_0]$,   by induction,  there exists an $i$ such that  $E_i \otimes_{F_1}M  \simeq  M  \times E_i''$ for 
 some $M$-algebra $E''_i$.  
 We have  $E_i \otimes_{F_0} M  \simeq E_i \otimes_{F_1}F_1\otimes _{F_0} M
 \simeq E_i \otimes_{F_1} (M  \times M) \simeq  E_i\otimes_{F_1}M  \times E_i\otimes_{F_1}M  \simeq  M  \times E''_i \times E_i \otimes_{F_1}M$. Hence $E_i \otimes_{F_0}M \simeq M \times E_i'$  for some $E'_i$.  
\end{proof}

\section{Corestriction of Cyclic Extensions over Quadratic Extensions}
\label{cores}
In this section we realize cyclic extensions over quadratic extensions with corestriction   zero  as 
dihedral extensions. 

Let $K$ be a field and $A$ a Galois module over $K$. For $n \geq 0$, let $H^n(K, A)$ denote the $n^{\rm th}$ Galois cohomology group
with values in $A$ (cf. \cite[Ch. VI]{Neukirch}). 
 For an extension of fields $M/K$, let res$ = $ res$_{M/K}  : H^n(K, A) \to H^n(M, A)$  be  the restriction homomorphism and
 for a finite extension $L/K$, let cores $  = $ cores$_{L/K}  : H^n(L, A) \to H^n(K, A)$ be the corestriction homomorphism (cf. \cite[p. 47]{Neukirch}).

Let  $F_0$ be  a  field and $F/F_0$ a   Galois  extension of degree 2. 
Let $\tau_0$ be the nontrivial automorphism of 
$F/F_0$.   Let $\overline{F}$ be an  algebraic closure of $F$.   
 Let $\tilde{\tau} \in $ Gal$(\overline{F}/F_0)$ be  such that $\tilde{\tau}$ restricted to $F$ is $\tau_0$.  
 Since $\tilde{\tau} \not\in  $ Gal$(\overline{F}/F)$ and $[F  : F_0] = 2$, 
  Gal$(\overline{F}/F_0) = $  Gal$(\overline{F}/F) \cup  $ Gal$(\overline{F}/F) \tilde{\tau}$
 and $\tilde{\tau}^2 \in$ Gal$(\overline{F}/F)$. 
 Let  Hom$_c$(Gal$(\overline{F}/F), \Z/m\Z)$ be the group of continuous homomorphisms from 
 Gal$(\overline{F}/F)$ to $\Z/m\Z$ with profinite topology on Gal$(\overline{F}/F)$ and discrete topology on $\Z/m\Z$. 
 Since the action of Gal$(\overline{F}/F)$ on $\Z/m\Z$ is trivial, we have 
 $H^1(F, \Z/m\Z) \simeq$  Hom$_c$(Gal$(\overline{F}/F), \Z/m\Z)$.  The group Hom$_c($Gal$(\overline{F}/F), \Z/m\Z)$ also classifies 
 isomorphism classes of pairs $(E, \sigma)$ with $E/F$ a cyclic extension of degree dividing $m$ and $\sigma$ a generator of Gal$(E/F)$.

 \begin{lemma}
 \label{cores-cal}
Let $\phi \in $ Hom$_c($Gal$(\overline{F}/F), \Z/m\Z)$.  
Then cores$(\phi) : $ Gal$(\overline{F}/F_0) \to \Z/m\Z$ is the homomorphism given by 
cores$(\phi)(\theta) = \phi(\theta) + \phi(\tilde{\tau}\theta \tilde{\tau}^{-1} )$ for all $\theta \in $ Gal$(\overline{F}/F)$ 
and cores$(\phi)(\tilde{\tau}) = \phi(\tilde{\tau}^2)$. 
 \end{lemma}
 
 \begin{proof} See   \cite[p 53]{Neukirch}.
 \end{proof}
 
 \begin{prop}
\label{galois} (cf. \cite[Proposition 24]{HKRT})
Let  $F_0$ be  a  field   and $F/F_0$ a quadratic Galois  field extension.
Let $E/F$ be a cyclic extension of degree  $m$ and $\sigma$ a generator of Gal$(E/F)$.
Then  cores$_{F/F_0}(E, \sigma)$ is zero if and only if  $E/F_0$ is dihedral extension.   
 \end{prop}

\begin{proof}  Since $E/F$ is a cyclic extension with generator $\sigma$, 
we have an isomorphism $\phi_0 : $ Gal$(E/F) \to \Z/m\Z$ given by $\phi_0(\sigma^i) \to i \in \Z/m\Z$.
 Let $\phi : $ Gal$(\overline{F}/F) \to \Z/m\Z$  be the  composition   Gal$(\overline{F}/F) \to $    Gal$(E/F) \buildrel{\phi_0}\over{\to} \Z/m\Z$.  The pair 
  $(E, \sigma) $ corresponds to the element $\phi$ in   Hom$_c$(Gal$(\overline{F}/F), \Z/m\Z)$. 
  Then cores$(\phi) : $ Gal$(\overline{F}/F_0) \to \Z/m\Z$ is the homomorphism given by 
cores$(\phi)(\theta) = \phi(\theta) + \phi(\tilde{\tau} \theta \tilde{\tau}^{-1})$ for all $\theta \in $ Gal$(\overline{F}/F)$ 
and cores$(\phi)(\tilde{\tau}) = \phi(\tilde{\tau}^2)$ (cf. \ref{cores-cal}). 

Suppose cores$_{F/F_0}(E, \sigma)$ is the zero homomorphism. Then cores$(\phi) : $ Gal$(\overline{F}/F_0) \to \Z/m\Z$ is the zero homomorphism
Let $\theta \in $ Gal$(\overline{F}/F)$. 
Then  $0 = $ cores$( \phi)(\theta) =  \phi(\theta) + \phi(\tilde{\tau}\theta \tilde{\tau}^{-1})$  and hence  
$\phi(\tilde{\tau} \theta \tilde{\tau}^{-1}) = - \phi(\theta)$.  
Suppose $\theta \in $ Gal$(\overline{F}/E) \subseteq $ Gal$(\overline{F}/F)$. 
Since Gal$(\overline{F}/E)$ is the kernel of $\phi$,   $\phi(\tilde{\tau} \theta \tilde{\tau}^{-1} ) = - \phi(\theta) = 0$ and  hence 
$\tilde{\tau}\theta \tilde{\tau}^{-1} \in $ Gal$(\overline{F}/E)$. 
Since  Gal$(\overline{F}/F_0)$ is generated by Gal$(\overline{F}/F)$ and $\tilde{\tau}$, 
Gal$(\overline{F}/E)$ is a normal subgroup of Gal$(\overline{F}/F_0)$.  Hence $E/F_0$ is a Galois extension. 
 
Let us denote the restriction of $\tilde{\tau}$ to $E$ by $\tau$.
Since $\tau \sigma \tau^{-1} $ is identity on $F$ and $E/F$ is Galois, $\tau \sigma \tau^{-1} \in $ Gal$(E/F)$.
Let $\tilde{\sigma} \in $ Gal$(\overline{F}/F)$ with restriction to $E$ equal to $\sigma$.
Since $\phi(\tilde{\tau} \tilde{\sigma} \tilde{\tau} ^{-1})= - \phi(\tilde{\sigma}) = \phi(\tilde{\sigma}^{-1}) $, it follows that 
$\phi_0(\tau \sigma \tau^{-1} ) =  \phi_0(\sigma^{-1})$. Since $\phi_0$ is an isomorphism, 
$\tau \sigma \tau^{-1} = \sigma^{-1}$.  Since  $\phi(\tilde{\tau}^2) = $ cores$(\phi)(\tilde{\tau}) = 0$,
it follows that $\phi_0(\tau^2) = 0$.  Since $\phi_0$ is an isomorphism, $\tau^2 $ is the identity  on $E$.
Since Gal$(E/F_0)$ is generated by $\sigma$ and $\tau$, with $\sigma^m  = 1$, $\tau^2 = 1$ and $\tau \sigma \tau^{-1} = \sigma^{-1}$,
Gal$(E/F_0)$ is a dihedral group of order $2m$. 

Conversely, suppose Gal$(E/F_0)$ is a dihedral extension. Since the subgroup of Gal$(E/F_0)$ generated by $\sigma$ is of index 2,
Gal$(E/F_0)$ is generated by $\sigma$ and $\tau$  with $\tau^2 = 1$ and $\tau\sigma\tau^{-1} =\sigma^{-1}$. 
Since $\tau \neq \sigma^i$ for all $i$, $\tau$ is not identity on $F$. 
Let $\tilde{\tau}$ be an extension of $\tau$ to $\overline{F}$.
Then, we have cores$(\phi)(\theta) = \phi(\theta) + \phi(\tilde{\tau} \theta\tilde{\tau}^{-1})$ for all $\theta \in $ Gal$(\overline{F}/F)$
and cores$(\phi)(\tilde{\tau}) = \phi(\tilde{\tau}^2)$ (\ref{cores-cal}).
Let $\theta \in $ Gal$(\overline{F}/F)$. Since Gal$(E/F)$ is cyclic and generated by $\sigma$, 
$\theta$ restricted to $E$ is $\sigma^i$ for some $i$.  
Since $\tau \sigma^i \tau^{-1} = \sigma^{-i}$,    
$\theta \tilde{\tau}\theta \tilde{\tau}^{-1} \in $ Gal$(\overline{F}/E)$. 
Since  the kernel of $\phi$ is Gal$(\overline{F}/E)$,  
cores$(\phi)(\theta) = \phi(\theta) + \phi(\tilde{\tau} \theta\tilde{\tau}^{-1}) = 
\phi( \theta \tilde{\tau} \theta \tilde{\tau}^{-1}) = 0$ for all  $\theta \in $ Gal$(\overline{F}/F)$. 
Since $\tau^2$ is identity on $E$, $\tilde{\tau}^2 \in $ Gal$(\overline{F}/E)$ and hence
 cores$(\phi)(\tilde{\tau}) = \phi(\tilde{\tau}^2)= 0$.
 Since Gal$(\overline{F}/F_0)$ is generated by Gal$(\overline{F}/F)$ and $\tilde{\tau}$,
 cores$(\phi) = 0$.
\end{proof}

\begin{cor}
\label{cores-field-zero} Let $F_0$ be a field   and $F/F_0$ a quadratic Galois extension. 
Let $m  \geq 2$ be coprime to char$(F_0)$.  Suppose that $F$ contains a primitive $m^{\rm th}$ root of unity $\rho$.
Let $a \in F_0^*$.  Suppose that  $[F(\sqrt[m]{a}) : F] = m$. 
Let $\sigma$ be the automorphism of $F(\sqrt[m]{a})$ given by $\sigma(\sqrt[m]{a}) = \rho  \sqrt[m]{a}$. 
 Then cores$(F(\sqrt[m]{a}),\sigma)$ is  zero if and only if   $N_{F/F_0}(\rho) = 1$.   
\end{cor}

\begin{proof}   
The lemma follows from   (\ref{galois}) and (\ref{dihedral-norm1}). 
\end{proof}

\begin{lemma}
\label{norm12}
Let $F_0$ be a field of characteristic not 2 and $F/F_0$ a quadratic extension. 
Let $n  \geq 1$.  Let $\rho$ be a    $2^n$-th root of unity in $F$. 
Suppose that  $\sqrt{-1} \not\in F_0$.  Then $N_{F/F_0}(\rho) = \pm 1$. 
\end{lemma}

\begin{proof}  If $n = 1$, then $\rho = -1$ and hence $N_{F/F_0}(-1) = (-1)^2 = 1$.
Suppose $n \geq 2$.  
 Let $\tau_0$ be the nontrivial automorphism of $F/F_0$.  Since $\rho$ is a  ${2^n}$-th root of unity,  
   $\tau(\rho)$ is also ${2^n}$-th root of unity 
 and hence  $  \rho \tau(\rho) $  is a  ${2^n}$-th root of unity  in $F_0$.
Since $\pm 1$ are the only $2^n$-th roots of unity in $F_0$,  $N_{F/F_0}(\rho) =  \rho\tau(\rho)   = \pm 1$. 
\end{proof}

\begin{cor}
\label{cores-field-zero-2} Let $F_0$ be a field of characteristic not 2 and $F/F_0$ a quadratic extension. 
Let $n  \geq 2$. Suppose that $F$ contains a primitive $2^n$-th root of unity $\rho$
and $\sqrt{-1} \not\in F_0$. 
Let $a \in F_0^*$.  Let $1 \leq d \leq n$. Suppose that  $F[(\sqrt[2^d]{a}) : F] = 2^d$. 
Let $\sigma_d$ be the automorphism of $F(\sqrt[2^d]{a})$ given by $\sigma(\sqrt[2^d]{a}) = \rho^{2^{n-d}} \sqrt[2^d]{a}$. 
If $d < n$, then cores$_{F/F_0}(F(\sqrt[2^d]{a}), \sigma_d) $ is zero. Further  $N_{F/F_0}(\rho) = 1$  if and only if 
cores$(F(\sqrt[2^n]{a}),\sigma_n)$ is  zero.   
\end{cor}

\begin{proof}   Suppose $d < n$.  
Then  $N_{F/F_0}(\rho^{2^{n-d}})  = N_{F/F_0}(\rho^{2^{n-d-1}})^2 = 1$ (\ref{norm12}) and hence 
 cores$(E_d, \sigma_d) $ is zero  (\ref{cores-field-zero}). 
 
 Suppose $n = d$. Then, by (\ref{cores-field-zero}), cores$( F(\sqrt[2^n]{a}), \sigma_n)$ is zero if and only if  $N_{F/F_0}(\rho) = 1$. 
\end{proof}

\begin{lemma}
\label{cores-field-zero-odd} Let $F_0$ be a field of characteristic not 2 and $F/F_0$ a quadratic extension. 
Let $\ell$ be a  prime not equal to char$(F_0)$. 
Let $n  \geq 1$. Suppose that $F$ contains a primitive $\ell^n$-th root of unity $\rho$
and  $F_0$ does not contain any  nontrivial  $\ell^{\rm th}$ root of unity. 
Let $a \in F_0^*$. Suppose that  $[F(\sqrt[\ell^n]{a}) : F] = \ell^n$. 
Let $\sigma$ be the automorphism of $F(\sqrt[\ell^n]{a})$ given by $\sigma(\sqrt[\ell^n]{a}) = \rho \sqrt[\ell^n]{a}$. 
Then cores$_{F/F_0}(F(\sqrt[\ell^n]{a}), \sigma) $ is zero. 
\end{lemma}

\begin{proof}
Since $\rho^{\ell^n} = 1$, $N_{F/F_0}(\rho)^{\ell^n} = 1$.
Since $F_0$ has no  nontrivial  $\ell^{\rm th}$ root of unity,   $N_{F/F_0}(\rho) = 1$ and   by (\ref{cores-field-zero}),  cores$(E, \sigma) = 0$. 
\end{proof}

\section{Dihedral  Extensions Over  Local Fields}
\label{dihedral-local}

Let  $F_0$ be   a complete discrete valued field with residue field  $\kappa_0$.  
Let $E/F_0$ be a dihedral extension of degree $2m$ with $2m$ coprime to char$(\kappa_0)$.
Let   $F \subseteq E$ be the fixed field of   the rotation subgroup of Gal$(E/F_0)$.
In this section we first  determine the degree of $E/F_0$ if $F/F_0$ is ramified and then we go on 
to describe all the dihedral extensions  of local fields.  

We begin with the following:

\begin{lemma}
\label{dihedral-cyclic}
Let  $F_0$ be   a complete discrete valued field with residue field  $\kappa_0$.  
Let $E/F_0$ be a dihedral extension of degree $2m$ with $2m$ coprime to char$(\kappa_0)$.
Suppose   the subfield $F$ of $E$ fixed by the rotation subgroup of Gal$(E/F_0)$  is ramified over $F_0$.
Let $L/F_0$ be an extension contained in $E$.  If    $F \not\subseteq L$ and $L/F_0$ is either unramified or totally ramified, 
then   $[L : F_0] \leq 2$. 
\end{lemma}

\begin{proof}
Let $L/F_0$ be an extension contained in $E$ with   $F \not\subseteq L$. 
We show that $L/F_0$ is cyclic. 
 
Suppose that $L/F_0$ is unramified. 
Let $\kappa$ be the residue field of $F$ and $\kappa'$ the residue field of $L$.
Then $\kappa = \kappa_0$ and $[\kappa' : \kappa_0] = [L : F_0]$.
Since $F/F_0$ is totally ramified, $LF/F$ is an extension of degree $[L : F_0]$ and the residue field of
$LF$ is also $\kappa'$.  Since $LF \subset E$ and $E/F$ is cyclic, $LF/F$ is cyclic.
In particular $\kappa'/\kappa_0$ is cyclic. Since $L/F_0$ is unramified and $F_0$ is complete,
$L/F_0$ is cyclic and by (\ref{subextndihedral-notgalois}), $[L : F_0] \leq 2$. 

Suppose that $L/F_0$ is totally ramified of degree $d$.  Since  $d$ is coprime to char$(\kappa_0)$, 
$L = F_0(\sqrt[d]{\pi })$ for some parameter $\pi  \in F_0$ (cf. \cite[Lemma 2.4]{PPS}).  
Since $F \not\subset L$, $[LF : F] = [L :F_0]$. 
Since  $E/F$ is cylic, $LF/F$ is cyclic. 
Since  $LF = F(\sqrt[d]{\pi})$ and $[LF : F] = [L : F_0] = d$,   $F$ contains a primitive $d^{\rm th}$ root of unity.
 Since $F/F_0$ is totally ramified, $F_0$ contains a primitive $d^{\rm th}$ root of unity. 
In particular $L/F_0$ is cyclic and by (\ref{subextndihedral-notgalois}), $[L : F_0] \leq 2$. 
\end{proof}

\begin{prop}
\label{dihedral-ram}
Let  $F_0$ be   a complete discrete valued field with residue field  $\kappa_0$.  
Let $E/F_0$ be a dihedral extension of degree $2m$ with   $2m$  coprime to char$(\kappa_0)$. 
If the subfield of $E$ fixed by the rotation subgroup of Gal$(E/F_0)$  is ramified over $F_0$, then  $[E : F_0]  \leq 4$.
\end{prop}

\begin{proof} 
 Let $F$ be the subfield of $E$ fixed  the rotation subgroup of Gal$(E/F_0)$.  Then $[F :F_0] = 2$. 
Suppose that $F/F_0$ is ramified.  Then $F/F_0$ is totally ramified.

Suppose that $[E : F_0]  = 2m \geq 5$.   
Suppose there is an odd prime $\ell$ dividing $m$.
Then there exists an extension   $L/F_0$   of degree $\ell$ such that $L \subset E$ and $F \not\subset L$ (\ref{cores1}). 
Since  $\ell$ is a  prime, $L/F_0$ is either unramified or totally ramified. 
Then, by   (\ref{dihedral-cyclic}), $[L : F_0] = \ell \leq 2$, 
leading to a contradiction.

Suppose there is no odd prime dividing $m$. Then $4$ divides $m$.
Thus there exists an extension   $L/F_0$   of degree $4$ such that $L \subset E$ and $F \not\subset L$ (\ref{cores1}). 
Since $[L : F_0] = 4$,   by   (\ref{dihedral-cyclic}),   $L/F_0$  is neither totally ramified nor unramified.  
Since $[L  : F_0] = 4$,  $L  = F_0(\sqrt{u})(\sqrt{\pi})$ 
for some $u \in F_0$ a unit and $\pi$ a parameter in $F_0(\sqrt{u})$.
Since $F/F_0$ is ramified, $F = F_0(\sqrt{\pi_1})$ for some parameter $\pi_1$ in $F_0$. 
Since $F_0(\sqrt{u})/F_0$ is unramified, 
$\pi_1$ is a parameter in $F_0(\sqrt{u})$ and hence  $\pi  = v \pi_1$ for some unit  $v \in F_0(\sqrt{u})$.
 Let $L' = F_0(\sqrt{u})(\sqrt{v})$.
Since $[LF : F_0] = 8$ and $LF = F_0(\sqrt{u})(\sqrt{v}, \sqrt{\pi_1})$, 
$[L' : F_0] = 4$. 
 Since  $L'/F_0$ is unramified, by (\ref{dihedral-cyclic}), $[L' : F_0]   \leq 2$, 
leading to a contradiction.
   \end{proof}
   
\begin{cor}
\label{cores-ram} 
Let  $F_0$ be   a complete discrete valued field with residue field  $\kappa_0$ of characteristic not 2. 
Let $F/F_0$ be a  ramified quadratic field extension. Let $E/F$ be a cyclic extension of degree coprime to 
char$(\kappa_0)$  and $\sigma$ a generator of 
Gal$(E/F)$. If cores$_{F/F_0}(E, \sigma)$ is zero, then $[E : F]  \leq 2$. 
\end{cor} 

\begin{proof} Suppose  cores$_{F/F_0}(E, \sigma)$ is zero.
Then $E/F_0$ is Galois with Gal$(E/F_0)$  dihedral (\ref{galois}). 
 Since  $F/F_0$ is ramified,  by (\ref{dihedral-ram}), $[E : F]  \leq 2$.
\end{proof}

\begin{prop} 
\label{totallyram} Let $K_0$ be a local field  and $L/K_0$  be  dihedral  extension of degree $2m$.
Let $K$ be the  subfield of $L$ fixed by the rotation subgroup of Gal$(L/K_0)$. If $K/K_0$ is unramified, then 
$L/K$ is totally ramified.
\end{prop}

\begin{proof}   
Let $L^{nr}$ be the maximal unramified sub extension of $L/K_0$.
Suppose that   $K/K_0$ is unramified. Then $K \subseteq L^{nr}$.
Suppose that $K \neq L^{nr}$. Then, by  (\ref{subextndihedral}), $L^{nr}/K_0$ is dihedral. 
Since $K_0$ is a local field and $L^{nr}/K_0$ is unramified, 
$L^{nr}/K_0$ is cyclic.  Since a dihedral group  can not be cyclic,  $L^{nr} = K$. 
 \end{proof}

\begin{remark}
\label{degree4}
Let $K_0$ be a local field with characteristic of the residue field not 2. 
Let $\pi \in K_0$ be a parameter and $u \in K_0$ a unit 
which is not a square. Since $K_0^*/K_0^{*2} = \{ 1, \pi, u, u\pi \}$ (cf. \cite[4.1, p. 217]{Sc}), 
$L = K_0(\sqrt{u} , \sqrt{\pi})$ is the unique degree four extension with Galois group $\Z/2\Z \times \Z/2\Z$.
Since $\Z/2\Z \times \Z/2\Z$ is the dihedral group of order 4,   $L/K_0$ is the unique   dihedral  extension  of degree 4.  
\end{remark}

\begin{theorem}
\label{dihedraleven}
 Let $K_0$ be a local field with characteristic of the residue field not 2 and  $\pi \in K_0$ be a parameter.  
Let  $d$  be   the maximum integer such that  $K_0(\sqrt{-1})$ contains a primitive $2^d$-th root of unity. 
Then there exists a dihedral extension of $K_0$ of degree $2^{n+1}$ with  $n \geq 2$ if and only if $\sqrt{-1} \not\in K_0$ and $n < d$. 
In this case $K_0(\sqrt{-1}, \sqrt[2^n]{\pi})$ is the unique   dihedral  extension of order $2^{n+1}$. 
\end{theorem} 

\begin{proof}  Suppose that $\sqrt{-1} \not\in K_0$ and  $2 \leq n < d$. 
Let   $\pi \in K_0$ be a parameter and $L_n  = K_0(\sqrt{-1}, \sqrt[2^n]{\pi})$.
Let $\rho \in K_0(\sqrt{-1})$ be a primitive $2^d$-th root of unity. Then $ \rho' = \rho^{2^{d-n}}$ is a primitive $2^n$-th root of unity and 
 $N_{K_0(\sqrt{-1})/K_0}(\rho') = (-1)^{2^{n-d}} = 1$ (cf. \ref{localfields}). Hence,   by  ( \ref{dihedral-norm1}),   $L_n/K_0$ is a dihedral extension.

 Suppose, conversely,  there exists a dihedral extension $L/K_0$ of degree $2^{n+1}$ with $n \geq 2$. 
Let $K$ be the subfield of $L$ fixed by the rotation subgroup of  Gal$(L/K_0)$.  
 Since $n \geq 2$, 
by (\ref{dihedral-ram}),   $K/K_0$ is  unramified.   By (\ref{totallyram}), $L/K$ is totally ramified. 
By (\ref{cores1}), there exists a subfield  $L'$ of $L$ with $[L' : K_0] = [L : K]$ and $L'K = L$ .
Since $K/K_0$ is unramified and $L/K$ is totally ramified, $L'/K_0$ is totally ramified.
Since the characteristic of the residue field of $K_0$ is not 2 and $[L': K_0 ] = [L : K] = 2^n$,
$L' = K_0(\sqrt[2^n]{\pi_0})$ for some parameter  $\pi_0 \in K_0$. Hence $L = K(\sqrt[2^n]{\pi_0})$. 

Suppose that $\sqrt{-1} \in K_0$.   Let $L_1 = K_0(\sqrt[4]{\pi_0}) \subset L'$.
Then $L_1/K_0$ is cyclic of degree 4, leading to a contradiction (\ref{subextndihedral-notgalois}). 
Hence $\sqrt{-1} \not\in K_0$.
 
 Since $L = K(\sqrt[2^n]{\pi_0})$ is a cyclic extension of  $K $ of degree $2^n$,
 $K$ contains a primitive $2^n$-th root of unity $\rho$. 
 Since $n \geq 2$, $\sqrt{-1} \in K$ and hence $K = K_0(\sqrt{-1})$.  Thus  by the maximality of $d$, $n \leq d$. 
  Suppose $n = d$.   
  Since $K_0$ is a local field,   by (\ref{localfields}), $N_{K/K_0}(\rho) = -1$.
 Hence, by (\ref{dihedral-norm1}),  $L/K_0$ is not dihedral, a contradiction. 
 Therefore $n < d$ and $L \simeq L_n$, proving the uniqueness of dihedral extensions of 
 degree $2^{n+1}$ over $K_0$.
 
 \end{proof}

\begin{theorem}
\label{dihedralodd}
Let $K_0$ be a local field with characteristic of the residue field not 2 and  $\pi \in K_0$ be a parameter.  
Let $\ell$ be an odd prime  not equal to  the characteristic of the residue field   of $K_0$. 
 Let $\rho$ be a primitive $\ell^{\rm th}$  root of unity and   $d$ be the  maximum integer such that 
$K_0(\rho)$ contains a primitive $\ell^d$-th root of unity. 
Then there exists a dihedral extension of $K_0$ of degree $2\ell^n$ with $n \geq 1$ if and only if 
$[K_0(\rho): K_0] = 2$ and $1 \leq n \leq  d$. 
In this case $K_0(\rho, \sqrt[\ell^n]{\pi})$ is the unique   dihedral  extension of order $2\ell^n$. 
\end{theorem} 

\begin{proof} Suppose  $[K_0(\rho): K_0] = 2$  and  $1 \leq n \leq d$.
  Let $\pi \in K_0$ be a parameter. 
 Let  $L_n = K_0(\rho, \sqrt[\ell^n]{\pi})$. 
 Let $\rho_n \in K_0(\rho)$ be a primitive $\ell^n$-th root of unity. 
 Since $N_{K_0(\rho)/K_0}(\rho_n)$ is an $\ell^n$-th root of unity  in $K_0$ and 
 the only $\ell^n$ root of unity in $K_0$ is 1, $N_{K_0(\rho) /K_0}(\rho_n) = 1$. 
 By (\ref{dihedral-norm1}),  $L_n/K_0$ is a dihedral extension.

Suppose, conversely  there exists a dihedral extension  $L/K_0$    of degree $2\ell^{n}$.  
Let $K$ be the subfield of  $L$ fixed by the rotation subgroup of Gal$(L/K_0)$.  
Since $[L : K ] = \ell^n \geq 3$,    by (\ref{dihedral-ram}), $K/K_0$ is   unramified.   
Then,   by (\ref{totallyram}), $L/K$ is totally ramified.
Let $L'$ be a subfield of $L$ with $[L' : K_0] = [L : K]$ and $L'K = L$ (\ref{cores1}).
Since $K/K_0$ is unramified and $L/K$ is totally ramified, $L'/K_0$ is totally ramified.
Since the characteristic of the residue field of $K_0$ is not $\ell$ and $[L : K_0 ] = [L : K] = \ell^n$,
$L' = K_0(\sqrt[\ell^n]{\pi_0})$ for some parameter $\pi_0 \in K_0$.  Hence $L = K(\sqrt[\ell^n]{\pi_0})$. 
Since $L/K$ is cyclic, $K$ contains a primitive $\ell^{n}$-th root of unity.  Thus $n \leq d$.
Since $n \geq 1$,   $\rho \in K^*$.  

Suppose $[K_0(\rho) : K_0] \neq 2$.  Since $K_0(\rho) \subseteq K$ and $[K : K_0] = 2$, 
$\rho \in K_0$.     Let $L_1 = K_0(\sqrt[\ell]{\pi_0})$. Then $L_1/K_0$ is cyclic.  
Since $K \not\subseteq L_1$, by (\ref{subextndihedral-notgalois}), $[L_1 : K_0]  = \ell \leq 2$, leading to a contradiction.  
Hence $[K_0(\rho) : K_0] = 2$. Since $\rho \in K$ and $[K : K_0] =2$,  $K = K_0(\rho)$. 
In particular $L\simeq L_n$, proving the uniqueness of dihedral extensions of degree $2\ell^n$ over $K_0$.
 \end{proof}
 
 \begin{cor}
 \label{dihedraln}
  Let $K_0$ be a local field with  the residue field $\kappa_0$ and  $m \geq 3$ with  $2m$ is   coprime to    char$(\kappa_0)$.  
 Let $L/K_0$  be   an   extension of degree $2m$ and $\pi \in K_0$ be a parameter. 
Then  $L/K_0$ is a dihedral extension  if and only if   there exists a primitive $ m^{\rm th}$ root of unity $\rho \in L$
with $[K_0(\rho) : K_0] = 2$, $N_{K_0(\rho)/K_0}(\rho) = 1$ and $L = K_0(\rho, \sqrt[m]{\pi})$. 
 \end{cor}

\begin{proof}   Suppose $L/K_0$ is a dihedral extension of degree $2m$.  
Suppose $m = 2^n$.  Let $d$ be the  maximum integer such that $K_0(\sqrt{-1})$ contains a 
primitive $2^d$-th root of unity. Then, by (\ref{dihedraleven}),  $\sqrt{-1} \not\in K_0$, $n < d$ and 
$L = K_0(\sqrt{-1}, \sqrt[2^n]{\pi})$.  Let $\rho \in K_0(\sqrt{-1})$ be a primitive $2^n$-th root of unity. 
Since $K_0(\rho)/K_0$ is unramified and  $K_0(\sqrt{-1})$ is the maximal unramified  extension  of $L/K_0$, 
$K_0(\sqrt{-1}) = K_0(\rho)$.  In particular $[K_0(\rho) : K_0] = 2$ and $L = K_0(\rho, \sqrt[2^n]{\pi})$.
Since  $L/K_0$ is dihedral,  by (\ref{dihedral-norm1}), $N_{K_0(\sqrt{-1})/K_0}(\rho) = 1$.

Assume that   there is an  odd prime dividing $m$. 
Let  $m = \ell_0^{n_0}\ell_1^{n_1} \cdots \ell_r^{n_r}$  with $\ell_0 = 2$,  for $i \geq 1$,
$\ell_i$ are  distinct odd primes,  $n_0 \geq 0$ and $n_i \geq 1$  for all $i \geq 1$.
For $0\leq i \leq r$, let $m_i = \frac{m}{\ell_i^{n_i}}$. 
Let $\sigma$ be a generator of the rotation subgroup of Gal$(L/K_0)$  and  $M_i  = L^{\sigma^{m_i}}$. 
Then $[M_i : K_0] = 2\ell_i^{n_i}$.

Let $1 \leq i \leq r$.  
Then $M_i/K_0$ is a dihedral extension of degree $2\ell_i^{n_i}$ with $\ell_i$ odd.
By  (\ref{dihedralodd}),   there exists a primitive $ \ell_i^{n_i}$-th root of unity $\rho_i \in M_i$, 
 $[K_0(\rho_i) : K_0] = 2$  and $M_i  = K_0(\rho_i, \sqrt[\ell_i^{n_i}]{\pi})$.
 Since $\ell_i$ are distinct primes and $M_i \subseteq L$, $\sqrt[m_0]{\pi} \in L$ and $\rho' = \rho_1 \cdots \rho_r \in L$ is a 
 primitive $m_0^{\rm th}$ root of unity.  
 If $n_0 = 0$, then $m_0 = m$.  Since  $\rho_i \not\in K_0$ for all  $i\geq 1$ and $\ell_i$'s are distinct primes, it follows that 
  $\rho' \not\in K_0$.  Since   $K_0(\rho')/K_0$ is an unramified extension and   $K_0(\sqrt[m]{\pi})/K_0$ is a totally ramified 
  extension of degree $m$, it follows that $[K_0(\rho') : K_0] = 2$ and $L = K_0(\rho', \sqrt[m]{\pi})$.   
  By (\ref{dihedral-norm1}), $N_{K_0(\rho')/K_0}(\rho') = 1$.

 Suppose $n_0  = 1$.  Then $M_0/K_0$ is the unique bi-quadratic extension  and hence 
 $M_0 = K(\sqrt{u}, \sqrt{\pi})$ (cf. \ref{degree4}). Suppose $n_0 \geq 2$. Then,  as in the first case,  
 $M_0$ contains a primitive $2^{n_0}$-th root of unity $\rho_0$, $[K_(\rho_0) : K_0] = 2$ 
  and  $M_0 = K_0(\rho_0, \sqrt[2^{n_0}]{\pi})$. 
 Hence in either case the maximal unramified extension of $M_0/K_0$ is of degree 2 over $K_0$. 
 
 Since $M_0 \subseteq E$, $\sqrt[2^{n_0}]{\pi} \in E$. Since $m = 2^{n_0}m_0$,   $\sqrt[m]{\pi} \in E$.    
 Since  $K_0(\sqrt[m](\pi))/ K_0$ is a totally ramified extension of degree $ m$,   the degree of the maximal 
 unramified extension of $L/K_0$ is 2.  Since $L$ contains a primitive $2^{n_0}$-th root of unity and 
 $m_0^{\rm th}$ root of unit,  $L$ contains a primitive $m^{\rm th}$ root of unity $\rho$.  Since $m \geq 3$, either $n_0 \geq 2$
 or $m_0 \geq 2$. Hence $[K_0(\rho) : K_0] = 2$ and $L = K_0(\rho, \sqrt[m]{\pi})$.  By (\ref{dihedral-norm1}), $N_{K_0(\rho)/K_0}(\rho) = 1$. 
 
Conversely, suppose  there exists a primitive $ m^{\rm th}$ root of unity $\rho \in L$
with $[K_0(\rho) : K_0] = 2$, $N_{K/K_0}(\rho) = 1$ and $L = K_0(\rho, \sqrt[m]{\pi})$.
Then, by (\ref{dihedral-norm1}), $L/K_0$ is  a dihedral extension.    
\end{proof}

We conclude this section with the following result on norms from dihedral extensions over local fields. 

\begin{prop}
\label{localfields-norms-n}
Let $K_0$ be a  local field  and $m \geq 2$ with $2m$ coprime to  the  characteristic the residue field of $K_0$. 
Let  $L/K_0$ be a dihedral extension  of degree $2m$.  Let $K$ be the subfield of $L$ fixed by the rotation subgroup of 
Gal$(L/K_0)$.  Let $L_0, \cdots, L_{m-1}$ be the subfields of $L$ with 
$L_iK = L$ and $[L : L_i] = 2$ (\ref{cores1}). 
Let   $ \theta_0 \in K_0^*$.  
Then  for every $0 \leq i \leq m-1$, 
 there exists $\mu_i \in L_i$,
such that $\prod_{i= 0}^{m-1} N_{L_i/K_0}(\mu_i) = \theta_0$. 
\end{prop}

\begin{proof}    Suppose $m = 2$. Then $L/K_0$ is a  biquadratic extension, 
 $L_0$ and $L_1$  are non isomorphic quadratic extensions of $K_0$. 
 Then, by local class field theory (cf. \cite[Proposition 3, p.142]{CFANT}), $N_{L_0/K_0}(L_0^*)$ and $N_{L_1/K_0}(L_1^*)$ 
 are two distinct subgroups of $K_0^*$ of index 2.
Let $b \in N_{L_0/K_0}(L_0^*)$ which is not in $N_{L_1/K_0}(L_1^*)$.
 Let $a \in K_0^*$.  Suppose $a \not\in N_{L_1/K_0}(L_1^*)$.
 Then $a \in b N_{L_1/K_0}(L_1^*)$. Hence $a = bc$ for some  $c \in N_{L_1/K_0}(L_1^*)$.
 In particular $a \in N_{L_0/K_0}(L_0^*) N_{L_1/K_0}(L_1^*)$.

 Suppose $m \geq 3$.   Let $\rho$ be a primitive $m^{\rm th}$ root of unity.  Then, for any parameter $ \pi \in K_0$,  
 by (\ref{dihedraln}), $ L = K_0(\rho, \sqrt[m]{\pi})$.  Let $\pi  \in K_0$ be a parameter. 
 Since $L = K_0(\rho, \sqrt{\pi})$  $[L : K_0(\sqrt[m]{\pi})] = 2$ and $K = K_0(\rho)$, 
 $K_0(\sqrt[m]{\pi}) = L_r$ for some $r$.  In particular $(-1)^m\pi$ is a norm from the extension $L_r/K_0$.
 Let $u \in K_0$ be a unit.   Since $u\pi$ is a parameter in $K_0$,  $\sqrt[m]{u\pi} \in L$ and 
  $K_0(\sqrt[m]{u\pi}) = L_s$ for some $s$.  Hence $(-1)^mu\pi$ is a norm from the extension $L_s/K_0$.
 In particular $u$ is a product of norms from  the extensions $L_r/K_0$ and $L_s/K_0$.
Since every element in $K_0$ is   $u \pi^r$ for some $u\in K_0$ a unit, it follows that 
every element in $K_0$ is a product of norms from the extensions $L_i/K_0$. 
\end{proof}

\section{Approximation of  norms from dihedral extensions over global fields}
\label{norms-globalfields}

\begin{prop}
\label{product-of-norms} Let $K_0$ be a  global  field   and   $n \geq 2$  an integer  with $2n$ coprime to  char$(K_0)$.   
 Let $E/K_0$ be a dihedral extension of degree $2n$ and $\sigma$ and $\tau$ generators of Gal$(E/K_0)$
 with $\sigma$ a rotation and $\tau$ a reflection.   Let  $K = E^\sigma$ and  $E_i = E^{\sigma^i \tau}$ for $1 \leq i \leq n$. 
  Let $\nu$ be a place of $K_0$ and $\lambda_\nu \in K_{0\nu}$.
  Suppose that the characteristic of the residue field at $\nu$ is coprime to $2n$.
   If $\lambda_\nu$ is a norm from the extension $E \otimes_{K_0}K_{0\nu}/K \otimes_{K_0}K_{0\nu}$, 
 then $\lambda_\nu$ is  a product of norms from the extensions $E_i \otimes K_{0\nu}/K_{0\nu}$.
\end{prop}

\begin{proof} Suppose $\lambda_\nu$ is a norm from the extension $E \otimes_{K_0}K_{0\nu}/K \otimes_{K_0}K_{0\nu}$. 

Suppose that $K\otimes_ {K_0} K_{0\nu}$ is not a field. Then $K\otimes_{K_0} K_{0\nu} \simeq K_{0\nu} \times K_{0\nu}$.
Since $KE_i = E$, $E \otimes_{K_0} K_{0\nu} \simeq E_i \otimes_{K_0} K \otimes_{K_0}K_{0\nu} \simeq E_i \otimes _{K_0} K_{0\nu} 
\times E_i\otimes_{K_0}K_{0\nu}$.  Since $\lambda_\nu$ is a norm from the extension $E \otimes_{K_0}K_{0\nu}/K \otimes_{K_0}K_{0\nu}$,    it follows that  $\lambda_\nu$ is a norm from $E_i \otimes_{K_0} K_{0\nu}/K_{0\nu}$. 

Suppose $K\otimes_{K_0}K_{0\nu}$ is a field. 
Suppose  $E\otimes_{K_0} K_{0\nu}   \simeq \prod_1^n K\otimes_{K_0} K_{0\nu}$. 
Then, by (\ref{ei-split}),  there exists an $i$ such that 
 $E_i  \otimes_{K_0} K_{0\nu} \simeq  K_{0\nu} \times E'_{i\nu}$ for some  $K_{0\nu}$-algebra
 $E'_{i\nu}$.   In particular $\lambda_\nu$ is norm from $E_i \otimes_{K_0} K_{0\nu}/K_{0\nu}$.

Suppose  $E\otimes_{K_0} K_{0\nu}$ is not isomorphic to    $\prod_1^n K\otimes_{K_0} K_{0\nu}$. 
Since $E/K_0$ is Galois, $E\otimes_{K_0}K_{0\nu}  \simeq \prod E_\nu$ for some field extension $E_\nu/K_{0\nu}$
and $K\otimes_{K_0}K_{0\nu}$ is a  proper subfield of   $E_\nu$.
Hence $E_\nu/K_{0\nu}$ is a dihedral extension. Since the characteristic of the residue field at $\nu$ is coprime to $2n$, 
  by (\ref{localfields-norms-n}), $\lambda_\nu$ is a product of norms from 
$E_i \otimes_{K_0} K_{0\nu}/K_{0\nu}$.  
\end{proof}

\begin{cor}
\label{norms-approximation}
Let $K_0$, $E$ and $K$ be as in (\ref{product-of-norms}).  
  Let  $S$ be a finite set of places of $K_0$ with $2n$ coprime to the characteristic of the residue field at places in $S$. 
  For $\nu \in S$, let  $\lambda_\nu \in K_{0\nu}$  be a  
   norm from the extension $E \otimes_{K_0}K_{0\nu}/K \otimes_{K_0}K_{0\nu}$. 
   Then there exists $\lambda \in K_0$ such that $\lambda$ is a norm from the extension $E/K$ and 
   $\lambda \lambda_\nu^{-1} \in K\otimes_{K_0}K_{0\nu}^{*n}$ for all $\nu \in S$.
      \end{cor}

\begin{proof} Let   $\sigma, \tau  \in $ Gal$(E/K_0)$ be as in  (\ref{product-of-norms}). 
 Let    $E_i = E^{\sigma^i \tau}$ for $1 \leq i \leq n$. Let $\nu \in S$. Then, by (\ref{product-of-norms}), 
 for  $1 \leq i \leq n$, there exists $z_{i\nu} \in E_i \otimes_{K_0}K_{0\nu}$ such that 
 $\lambda_\nu =  \prod N_{E_i \otimes_{K_0}K_{0\nu}/K_{0\nu}}(z_{i\nu})$.

   For $1 \leq i \leq n$, let $z_i \in E_i$ be   close to $z_{i\nu}$ for all $\nu \in S$. 
   Let $\lambda = \prod N_{E_i/K_0}(z_i)$.
   Since $z_i$ is close to $z_{i\nu}$ for all $\nu \in S$, 
   $\lambda$ is close to $\lambda_\nu$ for all $\nu \in S$. In particular 
   $\lambda\lambda_\nu^{-1} \in (K\otimes_{K_0}K_{0\nu})^{*n}$.
   Since $KE_i = E$, $\lambda$ is a norm from the extension $E/K$. 
\end{proof}

\section{Central Simple Algebras With Involutions Of Second Kind Over 2-local fields}
\label{alg-2-local}

In this section we give a description of central simples algebras having involutions of second kind over complete discretely 
valued fields with residue fields local fields (such fields are called 2-local fields).

 We begin the following well known 
 
 \begin{lemma}
 \label{e=f} Let $F$ be a complete discretely valued field and $\pi \in F$ a parameter.
 Let $E/F$ be a cyclic extension and $\sigma$ a generator of Gal$(E/F)$. 
 Then the cyclic algebra $(E, \sigma, \pi)$ is unramified if and only if $E = F$.
 \end{lemma}
 
 \begin{proof} Let $ m = [E : F]$.  Since $E/F$ is unramified,  the order of the class of $\pi$ in  $F^*/N_{E/F}(E^*)$
  is $m$  and hence $D = (E, \sigma, \pi)$ is a division algebra of degree $m$ (\cite[Theorem 6, p. 95]{Albert}).
 Let $\nu$ be the discrete valuation on $F$ and $\tilde{\nu}$ be the extension of $\nu$ to $D$ (\cite[Theorem 12.10, p. 138]{reiner}). 
Let  $e$ be  the ramification index of $D$.
 Since there exists $y \in D$ with $y^m = \pi$, we have $\tilde{\nu}(\pi) \geq m$ and hence $e = m$ (\cite[Theorem 13.7, p. 142]{reiner}).
 Suppose $D$ is unramified. Then $e = 1$  and hence $m = 1$. In particular 
$E = F$.  
 \end{proof}
 
\begin{lemma} 
\label{ramified2}
Let  $F_0$ be  a complete discrete valued field with residue field of characteristic not 2.
Let $\pi \in F_0$ be a parameter  and $F = F_0(\sqrt{\pi})$.
Let   $E/F$ be  an  unramified cyclic extension   and $\sigma$ a generator of Gal$(E/F)$.
If cores$_{F/F_0}(E, \sigma, \sqrt{\pi})$ is unramified,  then $(E, \sigma, \sqrt{\pi})$ is zero.
\end{lemma}

\begin{proof}  Let $E_0$ be the maximal unramified extension of $E/F_0$. 
Since $E/F$ is unramified and $F/F_0$ is ramified extension of degree 2, 
$E/E_0$ is of degree 2 and $E = E_0F$. Since $F/F_0$ is ramified, $E_0/F_0$ is unramified.
Since  $E/F$ is cyclic, $E_0/F_0$ is cyclic (cf. proof of \ref{dihedral-cyclic}).  Let $\sigma_0$ be the restriction of $\sigma$ to $E_0$.
Then $(E_0, \sigma_0) \otimes F = (E, \sigma)$. Hence cores$_{F/F_0}(E,   \sigma, \sqrt{\pi})  =  
(E_0, \sigma_0, N_{F/F_0}(\sqrt{\pi})) =  (E_0, \sigma_0,  -\pi)$ (\cite[Proposition 1.5.3]{Neukirch}). 

Suppose that cores$_{F/F_0}(E, \sigma, \sqrt{\pi})  = (E_0, \sigma_0,  -\pi)$ is unramified. 
Since $\pi$ is a   parameter in $F_0$ and $E_0/F_0$ is unramified, by (\ref{e=f}), 
 $E_0 = F_0$.  In particular $E = F$   and    $(E, \sigma, \sqrt{\pi}) $ is zero. 
\end{proof}
  
  \begin{lemma} 
\label{ramifiedh2}
Let  $F_0$ be  a complete discrete valued field with residue field  $K_0$
 and  $F/F_0$ a ramified quadratic  field extension. Let  $m \geq 1$ with $2m$ coprime to char$(K_0)$
   and  $\alpha \in H^2(F, \mu_m)$. 
 If  cores$_{F/F_0}(\alpha)$ is zero,  then $\alpha = \alpha_0 \otimes F$ for some 
  $\alpha_0 \in  H^2_{nr}(F_0, \mu_m)$.
 In particular per$(\alpha) \leq 2$. 
\end{lemma}

\begin{proof} 
Since $F/F_0$ is a ramified quadratic extension and char$(K_0) \neq 2$, 
 $F =F_0(\sqrt{\pi})$ for some $\pi \in F_0$ a parameter. 
 Since $m$ is coprime to char$(K_0)$,  
  $\alpha = \alpha' + (E, \sigma, \sqrt{\pi})$ for some $\alpha' \in H^2_{nr}(F, \mu_m)$ and 
 $E/F$ an unramified cyclic field extension of $F$  (cf. \cite[Lemma 4.1]{PPS}).  Since cores$_{F/F_0}(\alpha) = 0$, we have 
 cores$_{F/F_0}(-\alpha' )  = $ cores$_{F/{F_0}}(E, \sigma, \sqrt{\pi}) $.
 Since $\alpha'$ is unramified, cores$_{F/F_0}(-\alpha')$ is unramified  (cf. \cite[p. 48]{CTSB}) and hence cores$_{F/{F_0}}(E, \sigma, \sqrt{\pi}) $ is unramified.
 Thus, by (\ref{ramified2}),  $(E, \sigma, \sqrt{\pi})$ is zero and hence $\alpha = \alpha'$. 
  Since the residue field of $F$ and $F_0$ are equal and both $F$ and $F_0$ are complete, 
 it follows that $\alpha = \alpha' = \alpha_0 \otimes F$ for some $\alpha_0 \in H^2_{nr}(F_0, \mu_m)$.
 \end{proof}

  \begin{lemma}  
\label{unramified}
Let  $F_0$ be  a complete discrete valued field   and  $F/F_0$ an unramified quadratic   extension.  Let $\pi \in F_0$ be a parameter
and $m \geq 1$. Suppose $2m$ is  coprime to the characteristic of  the residue field of $F_0$. 
 Let $\alpha = \alpha' + (E, \sigma, \pi) \in H^2(F, \mu_m)$ for some $\alpha' \in H^2_{nr}(F, \mu_m)$ and 
 $E/F$ an unramified  cyclic field extension. 
 If   cores$_{F/F_0}(\alpha)$  is zero, then cores$_{F/F_0}(\alpha')$ and cores$_{F/F_0}(E, \sigma, \pi)$ are zero. 
\end{lemma}

\begin{proof} Since $F/F_0$ is  unramified  extension,   $\pi $ is a parameter in $F$. 
Since $\pi \in F_0$, we have cores$_{F/F_0}(E, \sigma, \pi) = $ cores$_{F/F_0}(E, \sigma) \cdot (\pi)$.
Since cores$_{F/F_0}( \alpha)  = $ cores$_{F/F_0}(\alpha') + $ cores$_{F/F_0}(E, \sigma, \pi) $ is zero, 
we have cores$_{F/F_0}(E, \sigma) \cdot (\pi) = - $cores$_{F/F_0}(\alpha')$. 
Since $\alpha'$ is unramified, cores$_{F/F_0}(\alpha')$ is unramified.
Since $E/F$ is unramified, cores$_{F/F_0}(E,\sigma)$ is unramified.
Since cores$_{F/F_0}(E, \sigma) \cdot (\pi)$ is unramified and $F_0$ is complete,   cores$_{F/F_0}(E, \sigma)$  is  zero (\ref{e=f}).  
Hence cores$_{F/F_0}(E, \sigma, \pi) = $ cores$_{F/F_0}(E, \sigma) \cdot (\pi)$ is zero and 
in particular   cores$_{F/F_0}(\alpha')$ is zero.  
\end{proof}

\begin{lemma}
\label{branch-alg}
 Let  $F_0$ be  a complete discrete valued field with residue field  $K_0$ a  local field, 
  $F/F_0$  a quadratic field extension  and $\pi  \in F_0$  a parameter. 
  Let $m \geq 1$  with $2m$  coprime to char$(K_0)$. 
  Let   $\alpha \in H^2(F, \mu_m)$ with   cores$_{F/F_0}(\alpha) = 0$. If  ind$(\alpha)\geq 3$, then $F/F_0$ is unramified and  $\alpha = (E, \sigma,  \pi)$ for some unramified cyclic extension $E/F$.
\end{lemma}

\begin{proof}  Suppose cores$_{F/F_0}(\alpha) = 0$ and  ind$(\alpha) \geq 3$.

Suppose $F/F_0$ is ramified. Then,   by (\ref{ramifiedh2}),  $\alpha$ is unramified and per$(\alpha) \leq 2$. 
Let $K$ be the residue field of $F$ and   $\beta \in H^2(K, \mu_m)$ be the image of $\alpha$. Since per$(\alpha) \leq 2$, per$(\beta) \leq 2$.
Since $K$ is a local field, ind$(\beta) = $ per$(\beta)$.
Since  $F$ is complete, ind$(\alpha) = $ ind$(\beta)$. Hence ind$(\alpha) \leq 2$, leading to a contradiction. 
Hence $F/F_0$ is unramified and  $\pi$  is a parameter in $F$. 
Since $m$ is coprime to char$(K_0)$, $\alpha = \alpha' + (E, \sigma, \pi)$ for some 
$\alpha' \in H^2_{nr}(F, \mu_m)$ and $E/F$ an unramified cyclic extension (cf. \cite[Lemma 4.1]{PPS}). 
 Then,  by (\ref{unramified}), cores$_{F/F_0}(\alpha')$ and cores$_{F/F_0}(E, \sigma, \pi)$ are zero. 
Let $\beta' \in H^2(K, \mu_m)$ be the image of $\alpha'$. 
Since  cores$_{F/F_0}(\alpha') = 0$, cores$_{\kappa/\kappa_0}(\beta') = 0$.
Since $K/K_0$ is a quadratic field extension of local fields,  $\beta' = 0$ ((cf. \cite[Theorem 10, p. 237]{LF}) ).
Since $F$ is complete, $\alpha' = 0$  and  hence  $\alpha = (E, \sigma, \pi)$.  
\end{proof}
 
 Let $F$ be a field and $m \geq 1$ coprime to char$(F)$. Suppose $F$ contains a primitive $m^{\rm th}$ root of unity $\rho$.
 For  $a, b \in F^*$, let  $(a, b)_{m}$ be the cyclic algebra generated by  $x$ and $y$ with relations $x^{m} = a$, $y^{m} = b$ and 
$yx = \rho xy$.

 \begin{prop}
\label{branch-alg-n} Let  $F_0$ be  a complete discrete valued field with residue field  $K_0$ a   local field.
Let $m\geq 3$ with    $2m$  coprime to the characteristic of the residue field of $K_0$. 
Let $\pi \in F_0$ be a parameter and $\delta \in F_0$
a unit such that the image of $\delta$ in $K_0$ is a parameter. 
Let $F/F_0$ be a quadratic field extension and $\alpha \in H^2(F, \mu_{m})$. If     cores$_{F/F_0}(\alpha) = 0$ and 
  ind$(\alpha)  = m$,   then  $F/F_0$ is unramified, $F$ contains a primitive $m^{\rm th}$ root of unity $\rho$, 
$N_{F/F_0}(\rho) = 1$ and 
  $\alpha = (\delta, \pi)_{m}$. 
\end{prop}
 
\begin{proof}  Suppose cores$_{F/F_0}(\alpha) = 0$ and ind$(\alpha)  = m$. 
  Since    $m \geq 3$, by (\ref{branch-alg}),  $F/F_0$ is unramified   and     $\alpha = (E, \sigma, \pi)$
for some  $E/F$ an unramified cyclic extension. 

Let  $K$ be the residue field of $F$ and $L$  the residue field of $E$. 
Since $F/F_0$ and $E/F$ are unramified,  $K/K_0$ is an extension of degree 2 and 
$L/K$ is a cyclic extension of degree $[E : F]$.
Let $\sigma_0$   denote the automorphism of $L/K$   induced by $\sigma$. 
Since cores$_{F/F_0}(E, \sigma) = 0$, cores$_{K/K_0}(L, \sigma_0) = 0$.
Hence, by (\ref{galois}), $L/K_0$ is a dihedral extension. 
Let $\overline{\delta} \in K_0$ be the image of $\delta$. Then, by the assumption, $\overline{\delta}$ is a
parameter in $K_0$. 
Since $K_0$ is a  local field and  $[L : K] =  m \geq 3$,   by (\ref{dihedraln}), $K = K_0(\rho)$ for a primitive $m^{\rm th}$ root of unity,
$N_{K/K_0}(\rho)= 1$  and $L = K_0(\rho,  \sqrt[m]{\overline{\delta}})$.  
Since $F_0$ is complete, $F = F_0(\rho)$ and  $E = F(\sqrt[m]{\delta})$. 
Since $N_{K/K_0}(\rho)= 1$, $N_{F/F_0}(\rho)= 1$. 
 Since   $F$ contains a primitive $m$-th root of unity, 
 $\alpha = (E, \sigma, \pi) = (\delta, \pi)_{m}$. 
\end{proof}

 We end this section with the following:

\begin{prop}
\label{prop-order}
Let  $F$ be  a complete discrete valued field with residue field  $K$, valuation ring  $R$,  $\pi \in R$ a parameter and
 $u \in R$ a unit.  Let $n \geq 2$ which is coprime to char$(K)$. Suppose that $F$ contains a primitive $n^{\rm th}$ root of unity
and  the cyclic algebra $D = (\pi, u)_n$ is a division algebra. Let $x, y \in D$ be  with $x^n = \pi$, $y^n = u $ and $xy = \rho yx$.
Then $ R[x] + R[x]y + \cdots  + R[x]y^{n-1} = R[y] + R[y]x + \cdots + R[y]x^{n-1} \subset D$ is the maximal order of $D$.
\end{prop}

 \begin{proof}
 Since $D$ is a division algebra and $F$ is complete, $\tilde{\nu} : D^* \to \Z$ given by 
 $\tilde{\nu}(z) = \nu(Nrd(z))$ is a discrete valuation on $D^*$ and 
 $\Lambda = \{ z \in D^* \mid \tilde{\nu}(z) \geq 0 \} \cup \{ 0\}$ is the unique maximal order of $D$ (\cite[Theorem 12.8]{reiner}).
 Since every element in  $R[y] + R[y]x + \cdots + R[y]x^{n-1}$ is integral over $R$,
$R[y] + R[y]x + \cdots + R[y]x^{n-1}  \subseteq \Lambda$.

Since  Nrd$(x) = (-1)^n \pi$,  $\tilde{\nu} (x) = n$.
 Since $y^n = u$ is a unit in $R$ and $n$ is coprime to the characteristic of $K$, 
 the extension $F(y)/F$ is unramified.  Since deg$(D)= n$, $[F(y) : F] = n$.
 Hence for any $a \in R[y]$, Nrd$(a) = N_{F(y)/F}(a)$.  
 Since $F(y)/F$ is unramified,  $\tilde{\nu}(a)$ is divisible by 
 $n$ for all $a \in R[y]$.    
Let $z\in  \Lambda$. Then 
   $z = \frac{1}{b}(a_0 + a_1 x + \cdots + a_{n-1}x^{n-1})$ for some $b \in R$, $b \neq 0$ and 
   $a_i \in R[y]$.  Since $\tilde{\nu}{a_i}$ is divisible by $n$ and $\tilde{\nu}(x) = 1$, 
   $\tilde{\nu}(a_0 + a_1 x + \cdots + a_{n-1}x^{n-1}) $ is equal to the minimum of   $\tilde{\nu}(a_ix^i)$ for $0 \leq i \leq n-1$. 
Since $\tilde{\nu}(z) \geq 0$, $\tilde{\nu}(a_0 + a_1x + \cdots + a_{n-1}x^{n-1}) \geq \tilde{\nu}(b)$. 
In particular $\tilde{\nu}(a_ix^i)  \geq \tilde{\nu}(b)$. Since $\tilde{\nu}(a_i x^i) = \tilde{\nu}(a_i) + i$ and $\tilde{\nu}(a_i)$,
 $\tilde{\nu}(b)$ are divisible by $n$, it  follows that $\tilde{\nu}(a_i) \geq \tilde{\nu}(b)$ for all $0 \leq i \leq n-1$.
 Hence $\frac{a_i}{b} \in R[y]$ and  $z \in R[y] + R[y]x + \cdots + R[y]x^{n-1}$. Hence 
 $\Lambda = R[y] + R[y]x + \cdots  +R[y]x^{n-1}$ is a maximal $R$-order of $D$. 
  \end{proof}

\section{Reduced unitary Whitehead  groups over 2-local fields}
\label{whitehead}

Let $F_0$ be a field of characteristic not equal to 2 and $F/F_0$ a quadratic \'etale extension.
Let $A$ be a central simple algebra over $F$  of degree $n$ with an involution $\tau$ of second kind and 
$F^{\tau} = F_0$.  Let $A^*$ denote the units in $A$ and   $S^+(A, \tau) = \{ z \in A^* \mid \tau(z) = z\}$.  
Let $\Sigma(A, \tau)$ be the subgroup of $A^*$ generated by $S^+(A, \tau)$.
Let $\Sigma'(A, \tau) = \{ z \in A^* \mid Nrd(z) \in F_0^* \}$. Then $\Sigma' (A, \tau)$ is a subgroup  of $A^*$ and 
$\Sigma(A, \tau) \subseteq \Sigma'(A, \tau)$ is a normal subgroup. The reduced unitary Whitehead group  $USK_1(A)$  of $A$
is the quotient group $\Sigma'(A, \tau)/\Sigma(A, \tau)$ (\cite[p. 267]{KMRT}).  

In this section we show that if $F_0$ is a 2-local field, then
$USK_1(A) $ is trivial.   This is an analogue of a theorem of Platonov on the triviality of $SK_1(A)$ over such fields 
 (\cite[Theorem 5.5]{Pla}).

\begin{prop}
\label{existence-of-involution} Let $F_0$ be a field  and $F/F_0$ a quadratic extension. 
Let $m  \geq 2$ with $2m$ coprime to char$(F_0)$. Suppose that $F$ contains a primitive $m$-th root of unity $\rho$.
Let $a, b \in F_0^*$. Suppose that  $F[(\sqrt[m]{a}) : F] = m$.  
Let $ A = (a, b)_{m}$ be the cyclic algebra generated by  $x$ and $y$ with relations $x^{m} = a$, $y^{m} = b$ and 
$yx = \rho xy$. 
Then  there exists a $F/F_0$-involution $\sigma$  on $A$ with $\sigma(x) = x$ and $\sigma(y) = y$    if and only if  $N_{F/F_0}(\rho) = 1$.
\end{prop}

\begin{proof} Let $\tau_0$ be the nontrivial automorphism of $F/F_0$.  Then $N_{F/F_0}(\rho) = \tau_0(\rho) \rho$.

Suppose  there exists a $F/F_0$-involution $\sigma$  on $A$ with $\sigma(x) = x$ and $\sigma(y) = y$.
Since $yx = \rho xy$, we have $xy = \sigma(yx) = \sigma(\rho xy) = \tau_0(\rho) yx = \tau_0(\rho) \rho xy$.
Hence $N_{F/F_0}(\rho) = \tau_0(\rho) \rho = 1$.

Suppose $N_{F/F_0}(\rho) = 1$.
Let  $c \in F_0^*$ with $F = F_0(\sqrt{c})$ and    $E = F(x)$. Then $A = E \oplus Ey \oplus \cdots \oplus Ey^{m-1}$.
Since $\{ x^iy^j, \sqrt{c}x^iy^j \mid 0 \leq i, j \leq m-1 \}$ is a $F_0$-basis of $A$, 
we  have a $F_0$-vector space isomorphism  $\sigma : A \to A$ given by 
$\sigma(x^iy^j) =  y^jx^i$ and   $\sigma( \sqrt{c}x^iy^j) =  - \sqrt{c}y^jx^i$.  Then 
   $\sigma(z) = \tau_0(z)$ for all $z \in F$. 
Since $a, b \in F_0$, $ \sigma(x^{m}) =   \sigma(a) = a = x^{m} = \sigma(x)^{m}$.
Similarly $\sigma(y^{m}) = b = \sigma(y)^{m}$.
Since  $\rho  = c_1 + c_2 \sqrt{c}$ for some $c_i \in F_0$ and $\tau_0(\sqrt{c}) = -\sqrt{c}$, we have 
$\sigma (\rho xy) = \sigma (c_1 xy + c_2 \sqrt{c} xy) = c_1 yx - c_2 \sqrt{c} yx = (c_1 - c_2\sqrt{c}) yx = \tau_0(\rho) yx$.
Since $yx = \rho xy$,   $\sigma(yx) = \sigma( \rho xy ) = \tau_0(\rho)  y  x = \tau_0(\rho) \rho xy = xy = \sigma(x)\sigma(y)$. 
Hence  $\sigma$ is a $F/F_0$-involution.
 \end{proof}

\begin{theorem} 
\label{branch-norms}  Let   $F_0$ be   a complete discrete valued field with residue field  $K_0$ a   local field.
Let $F/F_0$ be a quadratic field extension and $A$ a central simple algebra  of index  $m$  with $F/F_0$-involution $\tau$.
Suppose that  $2m$ is coprime to  the  characteristic of the residue field of $K_0$.
Then, every element in $ F_0^*$  is a  reduced norm  of   an element  in $\Sigma(A, \tau)$.  
\end{theorem}

\begin{proof}   Let $\alpha \in H^2(F, \mu_m)$ be the class of $A$. 
 Since $A$ admits a $F/F_0$-involution, cores$_{F/F_0}(\alpha) = 0$.
 Let $\lambda \in F_0$. Then cores$_{F/F_0}(\alpha \cdot (\lambda)) = $ cores$_{F/F_0}(\alpha) \cdot (\lambda) = 0$.
 Since $K_0$ is a local field, cores: $ H^3(F, \mu_n^{\otimes 2}) \to H^3(F_0, \mu_n^{\otimes 2})$ is injective (\cite[Proposition 4.6]{PPS})
 and  $\alpha \cdot (\lambda) = 0$.
By (\cite[Theorem 4.12]{PPS}), $\lambda$ is  a reduced norm from $A$.

Suppose   ind$(A) \leq 2$. Since every element of $F_0$ is a reduced norm from $A$,  
 by result of Yanchevskii (\cite{yan}, cf. \cite[Proposition 17.27]{KMRT}),  
every element in $F_0$  is the   reduced norm of an  element in $\Sigma(A, \tau)$.

Suppose that ind$(A)  = m \geq 3$.  Let $\pi \in F_0$ be a parameter and $\delta \in F_0$ a unit   such that 
the image $\overline{\delta}$  of $\delta$ in $K_0$ is a parameter. 
Then, by (\ref{branch-alg-n} ),  $F$ contains a primitive $m^{\rm th}$ root of unity $\rho$, $N_{F/F_0}(\rho) = 1$
and $A = (\delta, \pi)_n$.   Let $x, y \in A$ be such that $x^m = \delta$, $y^m = \pi$ and $yx = \rho xy$. 
Since $\pi, \delta \in F_0$ and $N_{F/F_0}(\rho) = 1$, by (\ref{existence-of-involution}), 
there exists an  $F/F_0$-involution $\tau'$ such that  $\tau'(x) = x$ and  $\tau'(y) = y$. 
Since the subgroup of $F_0^*$  consisting of   reduced norms of elements in  $\Sigma(A, \tau )$ 
does not depend on the Brauer class of $A$ 
 and  the involution $\tau$ (\cite[Proposition 17.24]{KMRT}), we assume that   $ \tau =  \tau'$. 
   
 Since $y  \in  \Sigma(A, \tau)$  and Nrd$(y)= (-1)^{m+1}\pi$,
  it is enough to show that  any  unit  in $F_0$   is the  reduced norm of an element in $\Sigma(A, \tau)$.
 
 Let $u \in F_0$  be    a unit.   
 Let $K$ be the residue field of $F$. 
 Let $E = F(\sqrt[m]{\delta})$ and $ L = K(\sqrt[m]{\overline{\delta}})$. 
 Then the residue field of $E$ is $L$. 
 Since $F$ contains a primitive $m^{\rm th}$ root of unity and $N_{F/F_0}(\rho) = 1$,
 $K$ contains a a primitive $m^{\rm th}$ root of unity  $\rho$ and $N_{K/K_0}(\rho) = 1$.
  Hence, by (\ref{dihedral-norm1}),  $E/F_0$ and $L/K_0$ are dihedral extensions.  
  
  By the choice of the involution $\tau$, the restriction of $\tau$ to $E$ is a reflection.   
  Let $\sigma$  be a generator of Gal$(E/F)$. 
  Since $E/F$ and $F/F_0$ are unramified, $E/F_0$ unramified. Since $L$ and $K_0$ are the residue fields of
  $E$ and $F_0$ respectively,  $\sigma$ and $\tau$ give rise to  elements in Gal$(L/K_0)$. Let us denote  the automorphisms of $L$ induced by 
  $\sigma$ and $\tau$  by $\sigma_0$ and $\tau_0$ respectively. 
  Then the dihedral group  Gal$(L/K_0)$ is generated by $\sigma_0$ and $\tau_0$ with $\tau_0$ a reflection. 
  Further for any $i$, $0 \leq i \leq m-1$, $E^{\tau\sigma^i}/F_0$ is unramified with residue field $L^{\tau_0\sigma_0^i}$.

  Let $\theta_0$ be the image of $u$ in $K_0$.  Then, by (\ref{localfields-norms-n}), 
  for $0 \leq i \leq m-1$, there exists $\mu_i \in L^{\tau_0\sigma_0^i}$ such that 
  $\prod_{i= 0}^{m-1} N_{L^{\tau_0\sigma_0^i}/K_0}(\mu_i) = \theta_0$. 
    
For each $0 \leq i \leq m-1$,   let   $\tilde{\mu}_i \in E^{\tau\sigma^i  }$  map to 
 $\mu_i$ in $L^{\tau_0\sigma_0^i}$.  Since $(\tau\sigma^i)^2 = id$ and the restriction of $\tau\sigma^i$ is nontrivial on $F$, 
 there exists an involution $\tau_i$ on $A$ with restriction to $E$ is $\tau\sigma^i$ (\cite[Theorem 10.1]{Sc}).
 In particular $\tilde{\mu}_i \in \Sigma(A, \tau_i)$. Since $\Sigma(A, \tau)  = \Sigma(A, \tau_i)$ (\cite[Proposition 17.24]{KMRT}), 
 $\tilde{\mu}  = \tilde{\mu}_0 \cdots \tilde{\mu}_{m-1} \in \Sigma(A, \tau)$.   Let $v = \prod_{i=0}^{m-1} N_{E^{\tau\sigma^i}/F_0}(\tilde{\mu}_i) 
 =$ Nrd$(\tilde{\mu})$. 
  Since the image of $v$ in $K_0$ is $\prod_{i= 0}^{n-1} N_{L^{\sigma^i\tau}/k_0}(\mu_i) = \theta_0$ and $\theta_0$ is the image of $u$,
  $uv^{-1} $ maps to 1 in $K_0$. Since  $m$ is coprime to char$(K_0) $, $uv^{-1} = w^{m}$ for some $w \in F_0$.
  Since $w \in S^+(A, \tau)$, $\mu w \in \Sigma(A, \tau)$. Since 
  Nrd$(\tilde{\mu} w) = v w^m = u$, the result follows.  
\end{proof}

\begin{cor} 
\label{whitehead-trivial}  Let   $F_0$ be   a complete discrete valued field with residue field  $K_0$ a   local field.
Let $F/F_0$ be a quadratic field extension and $A$ a central simple algebra  of index  $m$  with $F/F_0$-involution $\tau$.
Suppose that  $2m$ is coprime to  the  characteristic of the residue field of $K_0$.
Then, $USK_1(A) = \{ 1 \}$.  
\end{cor}

\begin{proof}  Let $z \in \Sigma'(A, \tau)$ and $\theta = $ Nrd$(z) \in F_0^*$.  Then, by (\ref{branch-norms}), there exists 
$\mu \in \Sigma(A, \tau)$ such that Nrd$(\mu) = \theta$. In particular Nrd$(z\mu^{-1}) = 1$,
Hence, by (\cite[Theorem 5.5]{Pla}), $z \mu^{-1}$ is a product of commutators in $A^*$. 
Since commutators are in $\Sigma(A, \tau)$ (\cite[Proposition 17.26]{KMRT}), $z \in \Sigma(A, \tau)$. Hence $USK_1(A) = \{1\}$. 
\end{proof}

\section{Reduced norms of central simple algebras over two dimensional  complete fields} 
\label{reducednorms-2dim}

Let $R$ be a complete regular local ring of dimension 2 with residue field $\kappa$ and field of fractions 
$F$.  For a prime $\theta \in R$,  let $F_\theta$ be the completion of $F$ at the discrete valuation given by the prime  ideal
$(\theta)$ of $R$ and $\kappa(\theta)$ the residue field at $\theta$. 
Let $A$ be a central simple algebra over $F$ of index coprime  to char$(\kappa)$. 
Let $m = (\pi, \delta)$ be the maximal ideal of $R$. 
Suppose that  $A$ is unramified on $R$ except possibly at $\pi$ and $\delta$. 
Let $\lambda    = v \pi^{s}\delta^t \in F^*$ for some unit $v \in R$ and 
$r, s \in \Z$.  In this section we show that if     $\kappa$ is a finite field  and 
$\lambda  \in$  Nrd$(A\otimes F_{\pi})$, then  $\lambda \in Nrd(A)$.

\begin{remark}
\label{branch-elements}
Let $\mu \in F^*_\pi$ and $n \geq 1$ coprime to char$(\kappa)$. Then $\mu =  u \pi^r$ for some $u \in F_\pi$ which is a unit at $\pi$. 
Let $\overline{u}$ be the image of $u$ in $\kappa(\pi)$. Since $\kappa(\pi)$ is the field of fractions of 
$R/(\pi)$ and $R$ is complete, $\kappa(\pi)$  is a complete discrete valued field with residue field $\kappa$
and $\overline{\delta} \in R/(\pi)$  is a parameter. Hence $\overline{u} = \overline{v} \delta^s$ for some $v \in R$ a unit.
Then, $\mu (v\pi^r\delta^s)^{-1}$ is a unit at $\pi$ and maps to 1 in $\kappa(\pi)$.
Since $n$ is coprime to $n$, $\mu = v\pi^r \delta^s c^n$ for some $c \in F_\pi^*$. 
\end{remark}

We begin by   extracting  the following from (\cite{PPS}). 
 
 \begin{prop} 
\label{2dim-local-nrd0}  Let $R$ be a complete regular local ring of dimension 2 with residue field $\kappa$ and field of fractions 
$F$.  Let $A$ be a central simple algebra over $F$ of index $n$  coprime  to char$(\kappa)$ and 
  $\alpha \in H^2(F, \mu_n)$  be  the class of $A$. 
Let $m = (\pi, \delta)$ be the maximal ideal of $R$. 
Suppose  that  $\kappa$ is a finite field and $A$ is unramified on $R$ except possibly at $\pi$ and $\delta$. 
Let $\lambda    = v \pi^{s}\delta^t \in F^*$ for some unit $v \in R$ and 
$r, s \in \Z$. 
If   $\alpha \cdot (\lambda) = 0  \in H^3(F, \mu_n^{\otimes 2})$, then    $\lambda \in Nrd(A)$. 
\end{prop}

\begin{proof} 
 As in (\cite[Theorem 4.12]{PPS}), we assume that ind$(A) = \ell^d$ with $\ell$ a   prime
and $F$ contains a primitive $\ell^{\rm th}$ root of unity. 
Since ind$(A)$ is coprime to char$(\kappa)$, $\ell \neq $ char$(\kappa)$. 
 We prove the result by induction on $d$.  If $ d = 0$, then $A$ is a matrix algebra and hence every element is a 
reduced norm from $A$. Suppose that $d \geq 1$. 

Suppose $\alpha \cdot (\lambda) = 0  \in H^3(F, \mu_n^{\otimes 2})$. 
Suppose $s$ is coprime to $\ell$. Then, by (\cite[Lemma 6.1]{PPS}), $A = (E, \sigma, (-1)^s \lambda)$ for some cyclic extension 
$E/F$ with $\sigma$ a generator of Gal$(E/F)$.  
In particular $(-1)^{\ell^d + 1} (-1)^s \lambda \in Nrd(A)$.  Suppose $\ell$ is odd, then $-1 \in Nrd(A)$ and hence 
$\lambda \in Nrd(A)$. Suppose $\ell = 2$. Since $s$ is odd,  $ \lambda = (-1)^{\ell^d + 1} (-1)^s \lambda \in Nrd(A)$.
Similarly if $t$ is coprime to $\ell$, then $\lambda \in Nrd(A)$.

Suppose that $s$ and $t$ are divisible by $\ell$. 
Then,  by (\cite[Lemma 4.10]{PPS}), there exists  an uramified cyclic field extension $L_\pi/F_\pi$   of degree $\ell$ and $\mu_\pi \in L_\pi$
 such that $N_{L_\pi/F_\pi}(\mu)  = \lambda$,    ind$(\alpha \otimes L_\pi) < $ ind$(A \otimes F_\pi)$ and $\alpha \cdot (\mu_\pi) = 0 \in H^3(L_\pi, \mu_{\ell^d}^{\otimes 2})$.

Since $L_\pi/F_\pi$ is an unramified cyclic extension  of degree $\ell$ and $F$ contains a primitive $\ell^{\rm th}$ root of
unity, we have $L_\pi = F_\pi(\sqrt[\ell]{a})$ for some $a \in F_\pi$  which is a unit  at $\pi$.
Since char$(\kappa) \neq \ell$ and the residue field $\kappa(\pi)$ of $F_\pi$ is the field of fractions of $R/(\pi)$, 
 we have $a = w\delta^\epsilon \in F_\pi^*/F_\pi^{*\ell}$ for some $w \in R$ a unit and 
$0 \leq \epsilon \leq \ell-1$  (cf. \ref{branch-elements}). Suppose $\epsilon \geq 1$.  Let $1 \leq \epsilon' \leq \ell-1$ with $\epsilon\epsilon' = 1$ modulo $\ell$. 
Since $F(\sqrt[\ell]{w\delta^{\epsilon}}) = F(\sqrt[\ell]{w^{\epsilon'}\delta})$, replacing $w$ by $w^{\epsilon'}$ we assume that    $ L  = F(\sqrt[\ell]{w\delta^\epsilon})$ with $0 \leq \epsilon \leq 1$. 
Then $L/F$ is a cyclic extension of degree $\ell$ and $L \otimes F_\pi \simeq L_\pi$.
Let $S$ be the integral closure of $R$ in $L$. Then $S$ is a regular local ring with maximal ideal 
$(\pi, \delta_1)$, where $\delta_1 = \delta$ or $\sqrt[\ell]{w\delta}$ depending on whether $\epsilon = 0$ or 1 (\cite[3.1, 3.2]{PS1}). 
Since ind$(\alpha \otimes L_\pi) < $ ind$(\alpha)$, by (\cite[Proposition 5.8]{PPS}), ind$(\alpha \otimes L) < $ ind$(\alpha)$.

Since $S$ is a regular local ring with maximal ideal $(\pi, \delta_1)$ and $L$ the field of fractions of $S$, 
there exists $u \in S$ a unit such that $\mu_\pi = u \pi^i\delta_1^{j} \mu_1^{\ell^{d}}$ for some $i, j \in \Z$ and $\mu_1 \in L_\pi$  (cf. \ref{branch-elements}).
Let $\mu' = u\pi^i\delta_1^{j} \in L$. Let $\lambda' = N_{L/F}(\mu') $. Then $\lambda' = v'\pi^{\ell i} \delta^{j'}$ for some unit $v' \in R$. 
Since $ \lambda = N_{L_\pi/F_\pi}(\mu_\pi) = N_{L_\pi/F_\pi}( \mu' \mu_1^{\ell^{d}}) =  \lambda'  N_{L_\pi/F_\pi}(\mu_1)^{\ell^{d}}$, 
we have
$\lambda' \lambda^{-1} \in F_{\pi}^{\ell^{d}}$. Hence, by (\cite[Corollary 5.5]{PPS}),  $\lambda = \lambda' \theta^{\ell^{d}}$ for some $\theta \in F$. 
Let $\mu = \mu' \theta^{\ell^{d}} \in L$.  Then $N_{L/F}(\mu) = \lambda$. 
Since $\mu = \mu' \theta^{\ell^{d}} = \mu_\pi \mu_1^{-\ell^{d+1}} \theta^{\ell^d}$  and  $\alpha \cdot (\mu_\pi) = 0 \in H^3(L_\pi, \mu_{\ell^d}^{\otimes 2})$,
 $\alpha \cdot (\mu) = 0 \in H^3(L_\pi, \mu_{\ell^d}^{\otimes 2})$.  Hence, by (\cite[Corollary 5.5]{PPS}), 
  $\alpha \cdot (\mu) = 0 \in H^3(L, \mu_{\ell^d}^{\otimes 2})$.  
  Since ind$(\alpha \otimes L) < $ ind$(\alpha)$, by induction $\mu \in $ Nrd$(A \otimes L)$.
  Since $N_{L/F}(\mu)= \lambda$,   $\lambda \in$ Nrd$(A)$. 
\end{proof}

 \begin{cor} 
\label{2dim-local-nrd}  Let $R$ be a complete regular local ring of dimension 2 with residue field $\kappa$ and field of fractions 
$F$.  
Let $A$ be a central simple algebra over $F$ of index coprime  to char$(\kappa)$. 
Let $m = (\pi, \delta)$ be the maximal ideal of $R$. 
Suppose  that  $\kappa$ is a finite field and $A$ is unramified on $A$ except possibly at $\pi$ and $\delta$. 
Let $\lambda    = v \pi^{s}\delta^t \in F^*$ for some unit $v \in R$ and 
$r, s \in \Z$. 
If   $\lambda  \in$  Nrd$(A\otimes F_{\pi})$, then  $\lambda \in Nrd(A)$. 
\end{cor}

\begin{proof}       Let $\alpha \in H^2(F, \mu_{\ell^d})$  be  the class of $A$. 
Since $\lambda \in $ Nrd$(A\otimes F_{\pi})$, $\alpha \cdot (\lambda) = 0 \in H^3(F_\pi, \mu_n)$. 
Since $\alpha$ is unramified on $R$ except possibly at $\pi$, $\delta$ and $\lambda = c\pi^s\delta^r$, 
$\alpha \cdot (\lambda) $ is unramified on $R$ except possibly at $\pi$ and $\delta$.
Since $\alpha \cdot (\lambda) = 0 \in H^3(F_\pi, \mu_n)$, by (\cite[Corollary 5.5]{PPS}), $\alpha \cdot (\lambda) = 0 \in H^3(F, \mu_n)$.
Hence, by  (\ref{2dim-local-nrd0} ), $\lambda  \in Nrd(A)$. 
\end{proof}

 We end this section with the following
 \begin{lemma} 
\label{2dim-local} 
Let  $v   \in R$  be a unit   and $\mu = v\pi_1^{r}\delta^s \in F^*$   for some
$r, s \in \Z$. 
Suppose that there exists $\theta_\pi \in F_{0\pi}^*$ such that 
 $\mu  \theta_\pi  \in  Nrd(A\otimes_F F_{\pi_1})$. Then there exists $\theta \in F_0$ 
 such that $\mu \theta  \in    Nrd(A)$  and $\theta_\pi \theta^{-1} \in F_{\pi_1}^n$. 
\end{lemma}

\begin{proof}  Since $F_0$ is the field of fractions of $R_0$ and $(\pi, \delta)$ is the maximal ideal of $R_0$, 
 $\theta_\pi = w \pi^{r_1} \delta^{s_1} c^n$ for some 
$w \in R_0$ a unit, $c \in F_{0\pi}$ (cf. \ref{branch-elements}). 
Let $\theta = w \pi^{r_1} \delta^{s_1} \in F_0^*$.
Since ind$(A) = n$,  $ c^n \in Nrd(A\otimes_F F_{\pi_1})$  and hence 
$\mu \theta \in   Nrd(A\otimes_F F_{\pi_1})$.     Since $A$ is unramified on $R$ except possibly at 
$\pi_1$ and $\delta$ and the support of  $\mu\theta$ is at most $\pi_1$ and $\delta$, by (\ref{2dim-local-nrd}), 
$\mu\theta \in Nrd(A)$. 
 \end{proof}

\section{Central simple algebras with involutions of second kind  over two dimensional  complete fields}  
\label{alg-2dimlocal}

Let $R_0$ be a  complete regular local ring  of dimension two with residue field 
$\kappa_0$ a finite field of characteristic not 2 and   $F_0$ the field of fractions of $R_0$. 
  Let $m = (\pi, \delta)$  be the maximal ideal of $R_0$.  
Let   $F/F_0$ be a quadratic field extension with $F = F_0(\sqrt{u\pi^\epsilon})$   for some  $u\in R_0$ a unit and $\epsilon = \{ 0,  1\}$. 
Let $R$ be the integral closure of $R_0$ in $F$. Then $R$ is a regular local ring with maximal ideal 
$(\pi_1 , \delta)$, where $\pi_1 = \pi$ if $\epsilon = 0$ and $\pi_1 = \sqrt{u\pi}$ if $\epsilon = 1$ (\cite[Lemma 3.1, 3.2]{PS1}).
Let $\kappa$ be the residue field of $R$. Then $[\kappa : \kappa_0] \leq 2$.

Let $A$ be  a central division  algebra over $F$    which      is unramified on $R$ except possibly at $\pi_1$ and $\delta$.  
Suppose that $n = $ ind$(A)$ is   coprime to char$(\kappa_0)$. 
In this section we show that if there is an involution $\tau$ on $A$ of second kind and $A$ is division, then  there exists a maximal $R$-order in $A$ invariant 
under $\tau$  with some additional structure. 
We  then  prove a   local global principle for certain  classes  of  hermitian forms over $(A, \tau)$.

We begin with the following

\begin{prop}
\label{2dim-alg} Let $\alpha \in H^2(F, \mu_n)$ be the class of $A$. 
Suppose    ind$(\alpha)  = n\geq 3$. 
 If cores$_{F/F_0}(\alpha)  = 0$,  then  $F/F_0$ is unramified on $R_0$, 
$F$ contains a primitive $n^{\rm th}$ root of unity $\rho$, $N_{F/F_0}(\rho) = 1$
and    $\alpha = (\delta, \pi)_{n}$.  
\end{prop}

\begin{proof} Since $F  = F_0(\sqrt{u\pi^{\epsilon}})$,  $R_0$ is complete and char$(\kappa) \neq 2$, 
it follows that   $F \otimes F_{0\pi} = F_{\pi_1}$ is a field.
Since $\alpha$ is unramfied on $R$ except possibly at $\pi_1$ and $\delta$,
ind$(\alpha) = $ ind$(\alpha \otimes F_{\pi_1})$ (\cite[Proposition 5.8]{PPS}).

 The residue field of $F_{0\pi}$ is a local field with residue field $\kappa_0$. 
Suppose  cores$_{F/F_0}(\alpha)  = 0$.  Then cores$_{F_{\pi_1}/F_{0\pi}}(\alpha) = 0$. 
Since ind$(\alpha \otimes F_{\pi_1} )  = n\geq 3$,  by (\ref{branch-alg-n}), 
$F_{\pi_1} /F_{0\pi}$ is unramified, $F_{\pi_1}$ contains a primitive $n^{\rm th}$ root of unity $\rho$, 
$N_{F_{\pi_1}/F_{0\pi}}(\rho) = 1$ and  $\alpha \otimes F_{\pi_1} = (\delta,\pi)_{n}$.
Since the residue field $\kappa(\pi_1)$ of $F_{\pi_1}$ is a  complete  discretely valued field with residue field $\kappa$,
$\kappa$ contains a primitive $n^{\rm th}$ root of unity.  Hence $F$ contains a primitive $n^{\rm th}$ root of unity.  
By (\cite[Corollary 5.5]{PPS}),    $\alpha = (\delta, \pi)_{n}$.
Since $F/F_0$ is unramified except possible at $\pi$ and $F_{\pi_1}/F_{0\pi}$ is unramified, $F/F_0$ is unramified on $R_0$.
Since $N_{F_{\pi_1}/F_{0\pi}}(\rho) = 1$, $N_{F/F_0}(\rho) = 1$. 
\end{proof}

Let $\alpha \in H^2(F, \mu_n)$ be the class of $A$. 
 Suppose    ind$(\alpha)  = n\geq 3$.  
 Since $(\pi, \delta) $ is a maximal ideal of $R$, $(\delta, \pi)_n$ is a division algebra.
 Let $D = (\pi, \delta)_n$. Then,     by (\ref{2dim-alg}),  $\alpha$ is the class of $D$. 
  Thus there exist   $x, y \in D$   such that 
$x^{n} = \delta$ and $y^{n} = \pi$ and $yx = \rho xy$.  
  Since $D \otimes F_\pi$ and $D \otimes F_\delta$ are division algebras (\cite[Proposition  5.8]{PPS}), 
the valuation $\nu_\pi$ and $\nu_\delta$  given by $\pi$ and $\delta$ on $F$
 extend to valuations $w_\pi$ and $w_\delta$ on $D \otimes F_\pi$
and $D \otimes F_\delta$ respectively (\cite[Theorem 12.6]{reiner}).  We have $ e_\pi : = [w_{\pi}(D^*) : \nu_\pi(F^*) ] = n$
and $e_\delta  : = [w_{\delta}(D^*) : \nu_\delta(F^*) ] = n$. 

Let $S = R[\sqrt[n]{\delta}] = R[x] $. Then $S$ is the integral closure of $R$ in $F(\sqrt[n]{\delta})$ and
$S$ is a regular local ring of dimension 2 with maximal ideal $(\sqrt[n]{\delta}, \pi)$ (\cite[Lemma 3.2]{PS1}).
Since  $D \simeq (\delta, \pi)_{n}$ and $N_{F/F_0}(\rho) = 1$, 
by (\ref{existence-of-involution}),
there exists an $F/F_0$-involution $\sigma$ on  $D$ with $\sigma(x) = x$  and $\sigma(y) = y$.

\begin{lemma}  (cf. \cite[Lemma 3.7]{Wu})
\label{order} Let $\Lambda = S + Sy + \cdots + Sy^{n -1} \subset D$. 
Then $\Lambda$ is a  maximal $R$-order  in $D$.  
\end{lemma}

\begin{proof}  Since $S$ is a free $R$-module, $\Lambda$ is a free $R$-module. 
Let $P \subset R$ be a height one prime ideal.
Suppose $P \neq (\pi)$ and $P \neq (\delta)$.  Since $\pi$ and $\delta$ are units at $P$, $\Lambda_P = \Lambda \otimes R_P$ is an 
Azumaya algebra  and hence a maximal  $R_P$-order in $D$.
Suppose $P = (\pi) $ or $(\delta)$. Then, by (\ref{prop-order}) and (\cite[Theorem 11.5]{reiner}),
 $\Lambda_P$ is a maximal $R_P$-order in $D$.
Since $R$ is noetherian,  integrally closed and  
$\Lambda$ is a reflexive $R$-module,  by (\cite[1.5]{AG1}),  $\Lambda$ is a maximal $R$-order of $D$.  
\end{proof}

\begin{lemma} (cf. \cite[Lemma 3.1]{Wu})
\label{rank1form}
Let $\sigma$ and $\Lambda$ be as above.  
 Let $a \in \Lambda$ with   $\sigma(a ) = a$.
If Nrd$(a) = u\pi^r\delta^s$ for some $u \in R_0$ unit and $r, s \in \Z$,
then  there exist   $\theta \in \Lambda$ a unit, 
$r', s' \in \{ 0, 1\}$ with $r \equiv r'$ and $s \equiv s'$ modulo 2 such that 
$<\!a\!> \simeq <\!  \theta  x^{r'}y^{s'}\!>$ as hermitian forms over $(D, \sigma)$.
 \end{lemma}

\begin{proof}  Let $r = 2r_1 + r'$ and $s = 2s_1 + s'$ with $r', s' \in \{0, 1 \}$. 
Let $z = x^{s_1}y^{r_1} \in \Lambda$. Then   Nrd$(z) = $ Nrd$(\sigma(z)) = \delta^{s_1}\pi^{r_1}$. 
 Let $\theta = \sigma(z)^{-1}  a  z^{-1} ( x^{s'}y^{r'})^{-1}$.
Then $a   = \sigma(z) \theta x^{s'}y^{r'} z$ and hence 
$<\!a \!> \simeq <\!\theta  x^{s'}y^{r'}\!>$.   Since Nrd$(\theta ) = u \in R_0$ is a unit, it follows,
 as in the proof of (\cite[Lemma 3.1]{Wu}), that  $\theta  \in\Lambda$  and  is  a unit in $\Lambda$. 
\end{proof}

\begin{cor}(cf. \cite[Corollary 3.2]{Wu})
\label{ranknforms}
Let $\sigma$ and $\Lambda$ be as in (\ref{order}).  
Let $h = <a_1,\cdots , a_r>$ be an hermitian form over $(A, \sigma)$
with   $a_i \in \Lambda$,  $\sigma(a_i) = a_i$ and  Nrd$(a_i) $ is a product of a unit in $R$, 
a power of $\pi$ and a power of $ \delta$.
Then 
$$   h \simeq <\!u_1, \cdots , u_{m_0}\!> \perp <\!v_1, \cdots, v_{n_1}\!>x \perp 
<\!w_1, \cdots, w_{n_2}\!>  y \perp <\!\theta_1, \cdots, \theta_{n_3}\!>xy$$ 
for some $u_i, v_i, w_i, \theta_i \in \Lambda$  units. 
\end{cor}

We have the following (cf. \cite[Corollary 3.3]{Wu}).
\begin{cor}
\label{isotropic}
Let $\sigma$ and $\Lambda$ as above.  Let $a_i \in\Lambda$   be as in (\ref{ranknforms})
and $h = <\!a_1, \cdots , a_r\!> $. 
If $h \otimes F_{0\pi}$ is isotropic, then $h$ is isotropic over $F_0$.   
\end{cor}

\begin{proof} Since $\sigma(xy) = yx = \rho xy$ and $\rho \sigma(\rho) = N_{F/F_0}(\rho) = 1$, it follows that 
 Int$(xy) \circ \sigma$ is an involution on $D$.
Following the proof of (\cite[Corollary 3.3]{Wu}), it follows that 
 if $h$ is isotropic over $F_{0\pi}$, then $h$ is isotropic over $F_0$. 
%
%
\end{proof}

\section{An application of refinement of patching to local global principle} 
\label{refinement}

Let $T$ be a complete discrete valuation ring and $K$ its field of fractions.  
We recall a few basic definitions from (\cite{HHK3}, \cite{HHK5}). 
Let $F$ be a function field of a curve over $K$. Let $\YY \to $ Spec$(T)$ be a  proper normal model of 
$F$ and $Y$ the special fibre.  For a  point $x$ of $Y$, let 
$F_x$ be the field of fractions of the completion $\hat{R}_x$  of the local ring at $x$.
Let $U$ be a nonempty proper  subset of  an irreducible component of $Y$ not containing the singular points of
$X$. Let $R_U$ be the subset of $F$ containing all those elements of $F$ which are regular at every closed point of $U$.
Let $t \in T$ be a parameter,  $\hat{R}_U$ be the $(t)$-adic completion of $R_U$ and $F_U$ the field of fractions of $\hat{R}_U$. 
 Let $P \in Y$ be a closed point. A height one prime ideal   $\pp$ of $\hat{R}_P$ containing $t$ is called a {\it branch} at $P$.   For a branch $\pp$,   
let $F_\pp$ be the completion of $F_P$ at the discrete valuation given by $\pp$.

Let $\PP$ be a  finite set of closed points of $Y$ containing all singular points of $Y$ and at least one point from each irreducible component of 
$Y$. Let $\UU $ be the set of irreducible components of $Y \setminus \PP$ and $\BB$ the set of branches at points in $\PP$. 
Let $G$ be a linear algebraic group over $F$. We say that  {\it factorization}  holds for $G$  with respect to $(\PP, \UU)$ 
if given $(g_{\pp}) \in \prod_{\pp \in \BB}G(F_{b})$, there exists $(g_Q) \in \prod_{Q \in \PP} G(F_Q)$
and $(g_{U}) \in \prod_{U \in \UU}G(F_{U})$ such that  if  $\pp$ is a branch at $P$ along $U$, then    $g_{\pp} = g_{Q} g_{U}$. 
If the  factorization  holds for $G$  with respect to  all possible pairs $(\PP, \UU)$, then we say that
{\it  factorization}  holds for $G$   over $F$ with respect to $\YY$. 
Let $Z$ be a variety over $F$ with a $G$-action. 
We say that $G$ {\it acts transitively} on points of $Z$ if $G(E)$ acts transitively on $Z(E)$ for all extensions $E/F$.

 Let $\XX \to \YY$ be a sequence of blow ups and   $X$  the special fibers of $\XX$. 
 Let  $P \in \YY$ be a closed point and  $V$ the fibre over $P$. 
Suppose that dim$(V)  = 1$.   
Let $\PP'$ be a  finite set of closed points of $V$ containing 
all the singular points of $V$ and at least one point from each irreducible component of $V$.
Let $\UU'$ be the set of connected components of $V \setminus \PP'$.
Let  $\BB'$ be the set of 
branches at the points of $\PP'$.   
We say that  {\it factorization}  holds for $G$  with respect to $(\PP', \UU')$ 
if given $(g_{\pp}) \in \prod_{\pp \in \BB'}G(F_{\pp})$, there exists $(g_Q) \in \prod_{Q \in \PP'} G(F_Q)$
and $(g_{U}) \in \prod_{U \in \UU'}G(F_{U})$ such that  if  $\pp$ is a branch at $P$ along $U$, then    $g_{\pp} = g_{Q} g_{U}$. 

Let $\PP$ be a finite set of  closed points of $X$ containing $\PP'$, 
all singular points of $X$ and at least one closed point from each irreducible component of $X$.
Let $\UU$ be the set of irreducible components of $X\setminus \PP$ and $\BB$ the set of branches at points in $\PP$.

The following results are immediate consequences   of results of  Harbater, Hartmann and Krashen  (\cite{HHK5}).

\begin{theorem}
\label{refinesha}
 Let $F$, $P$, $\PP$, $\PP'$, $\UU$ and $\UU'$ be as above.  
Let $G$ be a connected  linear algebraic group over $F$. If the   factorization holds for $G$   with respect to $(\PP, \UU)$,  
 then the  kernel of natural map 
$$H^1(F_P, G) \to \prod_{U' \in \UU'}H^1(F_{U'}, G) \times \prod_{Q \in \PP'}H^1(F_Q, G)$$ 
is trivial. 
\end{theorem}

\begin{proof}   Suppose the   factorization holds  for $G$  with respect to $(\PP, \UU)$. 
Then, by (\cite[Proposition 3.14]{HHK5}),  factorization holds  for $G$  with respect to $(\PP', \UU')$. 
 By (\cite[Proposition 3.10]{HHK5}), patching holds for the  injective diamond $F_{P_\bullet} = 
(F_P \leq \prod_{Q \in \PP'} F_Q, \prod_{U' \in \UU'}F_{U'} \leq \prod_{b' \in \BB'} F_{b'})$.
Hence,  by (\cite[Theorem 2.13]{HHK5}),  
 the map $H^1(F_P, G) \to \prod_{U' \in \UU'}H^1(F_{U'}, G) \times \prod_{Q \in \PP'}H^1(F_Q, G)$ is injective. 
\end{proof}

\begin{cor} Let $F$,  $\YY$, $P$ and  $\XX$ be   as above. Let  $G$ be a connected linear algebraic group over $F$. 
If the   factorization holds for $G$ over $F$, 
then the  kernel of natural map 
$$H^1(F_P, G) \to \prod_{x \in V}H^1(F_x, G)$$
is trivial. 
\end{cor}

\begin{proof} Let $\xi$ be in the kernel of the map $H^1(F_P, G) \to \prod_{x \in V}H^1(F_x, G)$.
Then, as in (\cite[Corollary 5.9]{HHK3}), there exists a finite set $\PP'$ of closed points of $V$ containing 
all the singular points of $V$ and at least one closed point from each irreducible component of $V$
such that if $\UU'$ is the set of irreducible components of $V \setminus \PP'$, 
then $\xi$ is in the kernel of $H^1(F_P, G) \to \prod_{U'  \in \UU}H^1(F_U, G) \times \prod_{Q  \in \PP'}H^1(F_Q, G)$.
Hence, by (\ref{refinesha}), $\xi$ is trivial.
\end{proof}
 
\begin{theorem} 
 \label{refinesha1} Let $F$, $P$, $F_P$ and $F_U$ be as above.
Let $G$ be a connected  linear algebraic group over $F$.
Suppose  the   factorization holds   for $G$ with respect to $(\PP, \UU)$. 
Let $Z$ be a  $F$-variety with $G$ acting transitively  on points $Z$.  
If $Z(F_{U'}) \neq \emptyset$ and $Z(F_Q) \neq \emptyset $ for all $U' \in \UU'$ and $Q \in \PP'$, then 
$Z(F_P) \neq \emptyset$. 
\end{theorem}

\begin{proof} By (\cite[Proposition 3.10]{HHK5}), patching holds for the  injective diamond $F_{P_\bullet} = 
(F_P \leq \prod_{Q \in \PP'} F_Q, \prod_{U' \in \UU'}F_{U'} \leq \prod_{b' \in \BB'} F_{b'})$.  Result follows from 
(\ref{refinesha}) and  (\cite[Corollary 2.8]{HHK3}). 
\end{proof}
 
 \begin{cor}
 \label{refinesha2}(cf. \cite[Theorem 9.1]{HHK3})
  Let $F$,  $\YY$, $P$, $\XX$ and $V$ be   as above. Let  $G$ be a connected linear algebraic group over $F$. 
Suppose the   factorization holds for $G$ over $F$. 
 Let $Z$ be a  $F$-variety with $G$ acting transitively  of points of $Z$. 
If $Z(F_x) \neq \emptyset$   for all $x \in V $, then 
$Z(F_P) \neq \emptyset$. 
\end{cor}

\begin{proof} Suppose $Z(F_x) \neq \emptyset$ for all $x \in V$.
Let $X_i$ be an irreducible component of $V$ and $\eta_i \in V$ the generic point of $X_i$.
Since $Z(F_{\eta_i}) \neq \emptyset$, by (\cite[Proposition 5.8]{HHK3}), 
there exists a  non-empty affine open subset $U_i$ of $X_i$  such that $Z(F_{U_i}) \neq \emptyset$.
Let $\PP'$ be the complement of the union of $U_i$'s in $V$. 
Let $Q \in \PP'$. Then, by the assumption on $Z$, $Z(F_Q) \neq \emptyset$.
Hence, by (\ref{refinesha1}), $Z(F_P) \neq \emptyset$.
\end{proof}

\section{Local global principle for projective homogeneous spaces under general linear groups}
\label{lgp-gl}

Let $K$ be a  complete discretely valued field  with residue field   $\kappa$   and $F$ the function field of a smooth projective 
 curve over $K$.   Let $A$ be a central simple algebra over $F$ of index $n$ coprime to char$(k)$ and $G = PGL(A)$. 
 Let $Z$ be a projective homogeneous space under $G$  over $F$.
 If $F$ contains a primitive $n^{\rm th}$ root of unity, then from the results in (\cite{RS}), it follows that 
 $Z(F) \neq \emptyset $ if and only if $Z(F_\nu) \neq \emptyset$ for all divisorial discrete valuations of $F$.
 In this section, we dispense with the condition on the roots of unity if $K$ is a local field. 
 
  Let $\XX $ be a normal  proper model of $F$ over the valuation ring of $K$. 
 Let $P$ be a closed point of $\XX$.  
 A discrete valuation  $\nu$ of $F$ (resp. $F_P$)   is called a {\it divisorial} discrete valuation  if 
 it is  given by a  codimension one point of a model of $F$ (resp. with center $P$).

Let $M$ be a field and $A$ a central simple algebra over $M$ of degree $n$. For a sequence of  integers 
$0 < n_1 < n_2 < \cdots < n_k < n$, let 
$$X(n_1, \cdots , n_k) = \{ (I_1, \cdots , I_k) \mid I_1 \subseteq I_2 \subseteq \cdots \subseteq I_k \subseteq A,$$  
$$ ~~~~~~~~~{\rm ~a ~
sequence~ of~ right~ ideals~ of~}  A {\it ~with} ~{\rm dim}_F(I_j)= n\cdot n_j, j = 1, \cdots , k \}.$$

\begin{theorem}
\label{proj-pgl} Let $K$ be a  local field   with residue field   $\kappa$   and $F$ the function field of a smooth projective  curve over $K$.   Let $A$ be a central simple algebra over $F$ of index coprime to char$(k)$. 
Let $\XX $ be regular proper model of $F$ over the valuation ring of $K$ and $P \in \XX$ be a closed point.
Let $L$ be the field $F$ or $F_P$.  Let $Z$  be a projective homogeneous space under $PGL(A)$ over $L$.
If $Z(L_\nu) \neq \emptyset$ for all divisorial  discrete valuation  $\nu$ of $L$, then $Z(L) \neq \emptyset$.  
\end{theorem}

\begin{proof} Let $f : \XX' \to \XX$  be a sequence of blow ups such that 
$\XX'$ is regular, 
the ramification locus of $A$ on $\XX'$  and the special fibre of $\XX'$  is a union of regular curves with 
normal crossings.    By blowing up $P$, we assume that the dimension of the fibre over $P$ is  1.
Let $V$ be either the special fibre of  $f$ or the fibre over $P$, depending on $L = F$ or $L  = F_P$.

 Suppose that $Z(L_\nu) \neq \emptyset$ for all  divisorial discrete valuations $\nu$ of $L$. 
 Since $PGL(A)$ is rational, factorization holds for $PGL(A)$ over $F$ (\cite[Theorem 3.6.]{HHK1}). 
 Thus, by (\cite[Theorem 5.10 and Theorem 9.1]{HHK3} for the case $L = F$ and \ref{refinesha2} for
  the case $L = F_P$), it is enough to show that $Z(F_x) \neq \emptyset$ for all 
 $x \in V$. 

Let $x \in V$. Suppose $x$ is a  generic point of $V$. Then $x$ defines a  divisorial discrete valuation $\nu_x$ of $L$ and 
 $L_x = L_{\nu_x}$.  Hence $Z(L_x) \neq \emptyset$. 
 
 Suppose $x = Q$  is a closed point of $V$.  
 Then, by the choice of $\XX'$,
the local ring $R_Q$ at $Q$ on $\XX'$ is generated by $(\pi, \delta)$ such that 
$A$ is unramified on $R_Q$ except possibly at $(\pi)$ and $(\delta)$.

 Let $n = $ deg$(A)$. 
Then $Z$ is isomorphic to  $X(n_1, \cdots , n_r)$ for some  sequence of integers 
$0 < n_1 < \cdots < n_r < n$ (cf. \cite[\S 5]{MPW1}).  Let $d$ be the  lcm of $n_1, \cdots , n_r, n$. 
Then,   for any field extension $M/F$, 
$Z(M) \neq \emptyset$ if and only if ind$(A\otimes_FM)$ divides  $d$. 
(cf. \cite[5.3]{MPW1}).

Let  $\nu_\pi$ be the discrete valuation given by $\pi$ on $L$ and $L_{\nu_\pi}$   the  completion of  $L$ at $\nu_\pi$.
Since $\nu_\pi$ is a divisorial discrete valuation of $L$, $Z(L_{\nu_\pi})  \neq \emptyset$.  
Hence ind$(A\otimes_FL_{\nu_\pi})$ divides $d$. 
 Since $L_{\nu_\pi} \subseteq L_{Q, \pi}$, ind$(A\otimes_F L_{Q, \pi})$ divides $d$. 
 Since,  by (\cite[Corollary 5.6]{PPS}), ind$(A\otimes_FL_Q) = $ ind$(A\otimes_FL_{Q, \pi})$,  ind$(A\otimes_F L_{Q})$ divides
 $d$. Hence $Z(L_Q) \neq \emptyset$. 
\end{proof}

\begin{cor}
\label{lgp-index} Let $K$ be a  local field   with residue field   $\kappa$   and $F$ the function field of a smooth projective  curve over $K$.   Let $A$ be a central simple algebra over $F$ of index coprime to char$(k)$. 
Let $\XX $ be normal  proper model of $F$ over the valuation ring of $K$ and $P \in \XX$ be a closed point.
Let $L$ be the field $F$ or $F_P$. 
Then ind$(A \otimes_F L ) =$ lcm$\{ {\rm ind}(A\otimes_{F}L_\nu) \mid \nu {\rm ~a ~divisorial ~discrete ~valuation~ of~} L \}$.
\end{cor}

\begin{proof}   Let $d = $ lcm$\{ {\rm ind}(A\otimes_{F}L_\nu) \mid \nu {\rm ~a ~divisorial ~discrete ~valuation~ of~} L \}$.
Then clearly  $d$ divides ind$(A\otimes_FL)$.
Thus  it is enough to show that ind$(A\otimes_FL)$ divides $d$.

Let  $Z = X(d)$.  Since for every divisorial  discrete valuation   $\nu$ of $L$, ind$(A\otimes_FL_\nu)$ divides $d$,
  $Z(F_\nu) \neq \emptyset$ (cf. \cite[5.3]{MPW1}). Hence, by (\ref{proj-pgl}), $Z(L) \neq \emptyset$.
  Thus  ind$(A\otimes_FL)$ divides $d$ (cf. \cite[5.3]{MPW1}).
\end{proof}

\begin{cor}
\label{proj-gl} Let $K$ be a  local field   with residue field   $\kappa$   and $F$ the function field of a smooth projective  curve over $K$.   Let $A$ be a central simple algebra over $F$ of index coprime to char$(k)$. 
Let $\XX $ be a regular proper model of $F$ over the valuation ring of $K$ and $P \in \XX$ be a closed point.
Let $L$ be the field $F$ or $F_P$.  Let $G = GL(A)$ and $Z$  be a projective homogeneous space under $G$ over $L$.
If $Z(L_\nu) \neq \emptyset$ for all divisorial  discrete valuation of $L$, then $Z(L) \neq \emptyset$.  
\end{cor}

\begin{proof}  Since the projective homogeneous spaces under $GL(A)$  are in bijection with the projective homogeneous spaces 
under $PGL(A)$ (\cite[Theorem 2.20(i)] {BT}), the corollary follows from (\ref{proj-pgl}).
\end{proof}

\section{Local global principle for homogeneous spaces under unitary groups}
\label{lgp-herm}

Let $K$ be a  local field with residue field   $\kappa$ of characteristic not 2  and $F_0$ the function field of a smooth projective 
 curve over $K$.  Let $F/F_0$ be a quadratic field extension.  
Let $A$ be a central simple algebra over $F$ with an involution $\sigma$ of second kind and 
$F^{\sigma} = F_0$. Let $(V, h)$ be an hermitian form over $(A, \sigma)$ and  $G = U(A, \sigma, h)$.
If ind$(A) = 1$, then  the validity of the conjecture 1  for $G$ is a consequence of results  proved in (\cite{CTPS}). 
If   ind$(A)  =  2$, Wu (\cite{Wu}) proved the validity of Conjecture 1 for $G$.   
In this section we dispense with the condition ind$(A) \leq 2$ for the good characteristic case. 

We begin by recalling  the structure of projective homogeneous spaces  under a unitary group over any field.
Let $F_0$ be a field and $F/F_0$ a separable quadratic extension. 
Let $A$ be a central simple algebra over $F$ of degree $n$  with an involution $\sigma$ of second kind and 
$F^{\sigma} = F_0$. Let $(V, h)$ be an hermitian form over $(A, \sigma)$ and  $G = U(A, \sigma, h)$.

Let $W$ be a finitely generated module over $A$. The {\it reduced dimension}  rdim$_A(W)$ of $W$ over $A$ is
defined as dim$_F(W)/n$ (\cite[Definition 1.9]{KMRT}).   For $0  < n_1 <  \cdots <  n_r \leq  n/2$ a sequence of integers and for any field extesnion
$L/F$, let 
$$
\begin{array}{lll} 
X(n_1, \cdots , n_r) &  = &  \{  (W_1, \cdots ,W_r) \mid \{ 0 \} \subsetneq W_1 \subsetneq \cdots \subsetneq W_r, W_i {\it ~a ~totally } \\
& & {\it ~~ isotropic ~subspace ~of} ~V ~{\it with ~rdim}_FW_i = n_i \}.
\end{array}
$$

We recall the following from (\cite{MPW1}, \cite{MPW2}, cf. \cite[\S 2]{Wu}). 

\begin{theorem}
\label{proj-hom-spaces}
Let $F_0$ be a field and $F/F_0$ a separable quadratic extension. 
Let $A$ be a central simple algebra over $F$ of degree $n$  with an involution $\sigma$ of second kind and 
$F^{\sigma} = F_0$. Let $(V, h)$ be an hermitian form over $(A, \sigma)$ and  $G = U(A, \sigma, h)$. Then \\
i) A  projective variety $X$  over $F_0$ is a projective homogeneous space  under $G$ over $F_0$
if and only if $X \simeq  X(n_1, \cdots , n_r)$  for some increasing sequence of integers $0 < n_1 <  \cdots < n_r \leq n/2$.\\
ii) For any field extension $L/F_0$,  $X(n_1, \cdots, n_r)(L) \neq \emptyset$ if and only if 
$X(n_r)(L) \neq \emptyset$ and ind$(A_L)$ divides $n_i$ for all $i$.\\
iii) If $A = M_r(D)$ for some central simple algebra over $F$ and $G_0 = SU(D, \sigma_0)$ for some unitary involution $\sigma_0$ 
on $D$, then there is a bijection assigning  projective homogeneous spaces  $X$ 
under $G$  and to  projective homogeneous spaces $X_0$ under $G_0$. Further for any field extension $L/F_0$,
$X(L) \neq \emptyset$ if and only is $X_0(L) \neq \emptyset$.
\end{theorem}

\begin{theorem}
\label{isotropic0}
Let $K$ be a local field with residue field $\kappa$.
Let $F_0$ be  the function field of  a curve over $K$.
Let $F/F_0$ be a quadratic extension and $A$ be a  central simple  algebra over $F$ 
 with an $F/F_0$- involution $\sigma$. 
 Suppose that $2$ ind$(A)$ is coprime to char$(\kappa)$. 
  Let $h$ be an hermitian form over $(A, \sigma)$ and $G = U(A, \sigma, h)$. 
  If $A = F$,  then assume that rank of $h$ is  at least 2. 
  Let $Z$ be a projective homogeneous space under $G$ over $F$. 
  Let $\XX_0$ be a normal proper model of $F_0$    and $P \in \XX_0$ be a closed point with 
  $F\otimes_{F_0}F_{0P}$ a field. 
 If $Z(F_{0\nu}) \neq \emptyset$  for  all divisorial   discrete valuations $\nu$ of $F_{0}$, then $Z( F_{0P}) \neq \emptyset$.  
\end{theorem} 

\begin{proof}    Let $n$ be the degree of $A$. 
Since $Z$ is a projective homogeneous space under $G$, by (\ref{proj-hom-spaces}), 
$Z \simeq X(n_1,\cdots , n_r)$ for some sequence of integers $0 < n_1 < \cdots < n_r \leq n/2$.
 Suppose that $Z(F_{0\nu}) \neq \emptyset$  for  all divisorial   discrete valuations $\nu$ of $F_{0}$. 
 Then, by (\ref{proj-hom-spaces}),  ind$(A\otimes_{F_0}F_{0\nu})$ divides $n_i$ for all $i$.
 Since    $K$ is a local field, it follows from \cite[Proposition 5.10]{PPS} and \cite[Theorem 5.5]{HHK1}) that 
 ind$(A)$ is the lcm of ind$(A\otimes_{F_0}F_{0\nu})$ as $\nu$ varies over all divisorial discrete valuations of $F$.
 Hence   ind$(A)$  divides $n_i$ for all $i$. Thus,  by (\ref{proj-hom-spaces}), 
 $X(n_r)(F_{0\nu})  \neq \emptyset$ for all all divisorial discrete valuations $\nu$  of $F$.
 To prove the theorem, by (\ref{proj-hom-spaces}),   it suffices  to show that 
 $X(n_r) (F_{0P}) \neq \emptyset$.  Thus we assume that  $ Z = X(m)$   with $m = n_r$. 

 Let $T$ be  the valuation   ring of $K$. Then there  exists a  sequence of blow ups   $\XX'_0 \to \XX_0 $    such that 
   the normalization $\XX$  of $\XX'_0$ in $F$ is regular and 
the ramification locus of $A$ on $\XX$ and the special fibre of $\XX$ is a union of regular curves with normal crossings
(\cite{A}, \cite{L}, cf. \cite{Wu}).  If necessary, by blowing up $P$, we assume that the fibre $V$ over $P$ is of dimension 1. 
Then, by (\ref{refinesha2}), it is enough to show that $Z(F_{0x}) \neq \emptyset$ for all $x \in V$. 

Let $x \in V$ be a generic point.  Then $x$ gives a  divisorial discrete valuation $\nu_x$ on $F_{0}$ 
such that $F_{0x} = F_{\nu_x}$.
Hence $Z(F_{0x}) \neq \emptyset$. 

Let $Q \in V$ be a closed point.  We show  that $Z(F_{0Q}) \neq \emptyset$ by  induction on ind$(A\otimes_{F_0}F_{0Q})$. 
Suppose   ind$(A \otimes_{F_0}F_{0Q} ) = 1$. Then the hermitian form  $h$ corresponds to a quadratic form over 
$q_h$ over $F_{0Q}$ such that $h$ is isotropic over any field extension  $M$ of $F_Q$ if and only if 
$q_h$ is isotropic over $M$ (\cite[Theorem 1.1, p. 348]{Sc}).  Since $Z = X(m)$,  for every divisorial discrete valuation $\nu$ of $F_0$, 
there is a totally isotropic subspace of $V\otimes _{F_0}F_{0\nu}$ of dimension $m$. 
Thus to prove the theorem,   it is enough to show that 
there is a totally isotropic subspace  of $V \otimes _{F_0} F_{0Q}$  of  dimension $m$.  
By induction on dim$(q_h)$, it is enough to show that $q_h$ is isotropic over $F_{0Q}$. 
By the assumption on the rank of $h$, rank of $q_h$ is at least 4.
Since   for every 
(divisorial) discrete valuation $\nu$ of $F_0$ centered on $Q$, $q_h$ is isotropic over $F_{0Q_\nu}$,   by (\cite[Corollary 4.7]{HHK5}), $q_h$ is isotropic over $F_{0Q}$.

Suppose  ind$(A \otimes_{F_0}F_{0Q} ) \geq 2$.
 Then by the choice of the model $\XX'_0$, we have   the  following:\\
i) the local ring $R_Q$ at $Q$ on $\XX_0'$ is regular with   $(\pi, \delta)$ as the maximal ideal,\\
ii) $F\otimes F_{0Q} = F_{0Q}(\sqrt{u\pi^{\epsilon}})$ for some unit $u \in R_Q$ and $\epsilon = 0, 1$, \\
iii) $A$ is unramified on the integral closure of $R_Q$ in $F_{0Q}(\sqrt{u\pi^{\epsilon}})$, except possibly at 
$\pi_1$ and $\delta$, where $\pi_1 = \pi$ or $\sqrt{u\pi}$ depending on $\epsilon = 0$ or 1.  

Let $D_Q$ be the central division algebra over $F \otimes_{F_0} F_{0Q}$  which is Brauer equivalent to $A \otimes_{F_0} F_{0Q}$.
The there is a unitary involution $\sigma_Q$ on $D_Q$ and   the hermitian form $(V, h)$  corresponds to a 
hermitian form  $(V_Q, h_Q)$ over $D_Q$.
By (\ref{proj-hom-spaces}), $Z_{F_{0Q}} = X(m)_{F_{0Q}}$ corresponds to a projective homogeneous space $Z_Q = X(m')$ under $U(D_Q, \sigma_Q)$ for some suitable $m'$ which is divisible by ind$(D_Q)$. 
Further to show that $Z(F_{0Q}) \neq \emptyset$, it is enough to show that $Z_Q(F_{0Q}) \neq \emptyset$.

Since ind$(A\otimes_{F_0}F_{0Q}) \geq 2$,   deg$(D_Q) \geq 2$. 
If deg$(D_Q)= 2$, let $\Lambda$ be the maximal $R_Q$-order of $D_Q$ as in (\cite[Lemma 3.7]{Wu}).
If deg$(D_Q) \geq 3$, let $\Lambda$ be the maximal $R_Q$-order of $D_Q$ as in (\ref{order}). 
Since $D_Q$ is a division algebra,  $h_Q \simeq <a_1, \cdots, a_n>$ for some $a_i \in \Lambda_P$.
Once again there exists a sequence of blow ups $\XX''_0 \to \XX'_0$ such that support of Nrd$(a_i)$ for all $i$
is a union of regular curves with normal crossings (\cite{A}, \cite{L}, cf. \cite[\S 4]{Wu}).  Further, by blowing up, we also assume that  $\XX''_0$ satisfies i), ii) and iii). 
Let $V''$ be the fibre over $Q$.   Once again we assume that dim$(V'') = 1$.   Thus to show that 
$Z_Q(F_{0Q}) \neq \emptyset$, by (\ref{refinesha2}), it is enough to show that 
$Z_Q(F_{0x}) \neq \emptyset$ for all $x \in V''$. 

Let $x'  \in V''$ be a  generic point, then, as above,  $Z_Q(F_{0x'}) \neq \emptyset$. 

Let $Q' \in V''$ be a closed point. 
Suppose that $D_Q \otimes_{F_0Q} F_{0Q'}$ is division. 
By (\ref{proj-hom-spaces}), it is  enough to show that 
there is a totally isotropic subspace  of $V_Q \otimes F_{0Q'}$  of  dimension $m'$.
By induction on the reduced dimension of $V_Q$, it is enough to show that 
$h_Q$ is isotropic over $F_{0Q'}$.  Since  $\XX_0''$ satisfies i), ii) and iii), 
the maximal ideal at $Q'$ on $\XX_0''$ is generated by $(\pi, \delta)$, 
$h_Q =  <a_1, \cdots, a_n>$ for some $a_i \in \Lambda_P$ with Nrd$(a_i)$ is a supported along only $\pi$, $\delta$, 
and  $h_Q\otimes F_{0Q'\pi}$ is isotropic. Hence, by (\ref{isotropic}), $h_Q$ is isotropic over $F_{0Q'}$. 

Suppose that $D_Q\otimes_{F_0Q}F_{0Q'}$ is not division. Then ind$(A\otimes_{F_0}F_{0Q'}) < $ ind$(A\otimes_{F_0} F_{0Q})$.
Hence, by induction, $Z_Q(F_{0Q'}) \neq \emptyset$. 
\end{proof}

\begin{theorem}
\label{isotropic1}
Let $K$ be a local field with residue field $\kappa$.
Let $F_0$ be a function field of  a curve over $K$.
Let $F/F_0$ be a quadratic extension and $A$   a central simple  algebra over $F$ of index $n$
 with an $F/F_0$- involution $\sigma$. 
 Suppose that $2n$ is coprime to char$(\kappa)$. 
 Let $h$ be an hermitian form over $(A, \sigma)$. 
  If $A = F$,  then assume that rank of $h$ is  at least 2. 
  Let $Z$ be a projective homogeneous space under $U(A, \sigma, h)$  over $F$. 
 If $Z(F_{0\nu}) \neq \emptyset$  for  all (divisorial)   discrete valuations $\nu$ of $F_{0}$, then $Z( F_{0}) \neq \emptyset$.  \end{theorem} 

\begin{proof}  Suppose 
 $Z(F_{0\nu}) \neq \emptyset$  for  all (divisorial)   discrete valuations $\nu$ of $F_{0}$. 
  Let $\XX_0$ be a normal proper model of $F_0$  over the valuation ring  of $K$  such that 
  the normal closure of $\XX_0$ and $X_0$ the 
  special fibre. 
  Let $x \in X_0$ be  a  codimension 0  point in $X_0$. Then $F_{0x} = F_{0\nu_x}$ for the discrete valuation 
  $\nu_x$ of $F_0$  given by $x$. Hence $Z(F_{0x}) \neq \emptyset$. 
 
  Let $P \in X_0$ be   a closed point. 
  We have $A \otimes_{F_0}F\simeq A_1 \times A_1^{\rm op}$ for some central simple algebra $A_1/F$ (\cite[Proposition 2.14]{KMRT})
   and
  $U(A, \sigma, h)\otimes F \simeq GL(A_1)$ (\cite[p. 346]{KMRT}). 
  Suppose $F \otimes_{F_0}F_{0P}$ is not a field. 
  Then  $F \subset F_{0P}$ and $A\otimes_{F_0}F_{0P} \simeq A_1 \otimes F_{0P}  \times A_1^{\rm op} \otimes {F_0}$.
  Let $\XX$ be the normal  closure of $\XX_0$ in $F$.  
   Since $F\otimes_{F_0}F_{0P}$ is not a field, there exists a closed point $Q$ of $\XX$ such that 
   $F_Q \simeq F_{0P}$.     
    Hence, by (\ref{proj-gl}),  $Z(F \otimes_{F_0}F_{0P}) \neq \emptyset$.
    
  Suppose $F \otimes_{F_0}F_{0P}$ is a field.  
  Then, by (\ref{isotropic0}),   $Z( F_{0P}) \neq \emptyset$. 
  Since $U(A, \sigma, h)$  is rational and connected   (\cite[p. 195, Lemma 1]{Merkurjev}),    by (\cite[Corollary 6.5 and Theorem 9.1]{HHK3}), $Z(F_0) \neq \emptyset$.  
\end{proof}

\begin{theorem} 
\label{uas}
Let $K$ be a local field with residue field $\kappa$.
Let $F_0$ be a function field of  a curve over $K$.
Let $F/F_0$ be a quadratic extension and $A$ a central simple algebra over $F$ of index $n$
 with an $F/F_0$- involution $\sigma$. 
 Suppose that $2n$ is coprime to char$(\kappa)$.
 Let $h$ be an hermitian form over $(A, \sigma)$. 
 Then the canonical map   $$H^1(F_0, U(A, \sigma, h) ) \to \prod_{\nu \in \Omega_{F_0}} H^1(F_{0\nu}, U(A, \sigma, h))$$
 has trivial kernel. 
\end{theorem}

\begin{proof}  Let 
$\xi \in H^1(F_{0}, U(A, \sigma, h))$.  Then   $\xi$ corresponds to a 
hermitian spaces $h'$  over $(A, \sigma)$ of  reduced rank  equal to the  reduced rank of $h$.
Let $h_0 = h\perp -h'$ and  $m$ the reduced rank of $h$. 
Let $G = U(A,\sigma, h_0)$ and $Z = X(m)$.
Then $Z$ is a projective homogeneous variety under $G$ over $F_0$.
Suppose $\xi$ maps to the trivial element in $H^1(F_{0\nu}, U(A,\sigma, h))$ for all  (divisorial) discrete valuations 
$\nu$  of $F_0$.  Then $h' \otimes F_{0\nu} \simeq h\otimes F_{0\nu}$ and hence $h_0$ is hyperbolic.
Thus $Z(F_{0\nu}) \neq \emptyset$ for all  (divisorial) discrete valuations 
$\nu$  of $F_0$.  Hence, by (\ref{isotropic1}), $Z(F_0) \neq \emptyset$.
In particular $h_0$ is hyperbolic. Since the reduced rank of $h$ and $h'$ are equal, 
$h \simeq h'$ and $\xi$ is the trivial element in $H^1(F_0, U(A, \sigma, h))$.
\end{proof}

\section{Local global principle for special unitary groups - patching setup}
\label{lgp-patching}

Let $F_0$ be a field of characteristic not equal to 2 and $F/F_0$ a quadratic \'etale extension.
Let $A$ be a central simple algebra over $F$  of degree $n$ with an involution $\sigma$ of second kind and 
$F^{\sigma} = F_0$.   
Then we have an  exact sequence of algebraic groups 
$$ 1 \to SU(A, \sigma) \to U(A, \sigma) \to R^1_{F/F_0}(\Gm) \to 1.$$
For any  field extension $L/F_0$, we have an 
induced  exact sequence 
$$U(A, \sigma)(L) \to (L\otimes_{F_0}F)^{*1} \to H^1(L, SU(A, \sigma)) \to H^1(L, U(A, \sigma)) \hskip 1.5cm (\star)$$
where $(L\otimes_{F_0}F)^{*1} = R^1_{F/F_0}(\Gm)(L)$ is the subgroup of $(L\otimes_{F_0}F)^*$ consisting of norm one elements and the map $U(A,\sigma)(L) \to 
(L\otimes_{F_0}F)^{*1}$ is given by the reduced norm. 
Further  the image of 
 $U(A, \sigma)(L) \to (L\otimes_{F_0} F)^{*1}$ is equal to 
$\{ \theta^{-1} \sigma(\theta) \mid \theta \in Nrd(A\otimes_{F_0}L^*) \} $ (\cite[p. 202]{KMRT}).

Let $K$ be a  local field   with the residue field $\kappa$  and $F_0$ the function field of a smooth projective 
 curve over $K$. Let $F/F_0$ be a separable quadratic extension.
   Let $A$ be a central simple algebra over $F$  of degree $n$ with an involution $\sigma$ of second kind and 
$F^{\sigma} = F_0$. Suppose that $2n$ is coprime to char$(\kappa)$.  
In this section we show that there is a local global principal for principal homogeneous spaces under $SU(A, \sigma)$ over $F_0$ 
in the patching setup  (cf. \ref{patching-su}).

Let $\mu \in F^*$. Let  $F= F_0(\sqrt{d})$,  $d \in F_0^*$.  
 Let $T$ be the valuation  ring   of  $K$.  
 Then there exists a regular proper model   $\XX_0 \to Spec(T)$ of $F_0$ with the  
 normalization $\XX$ of $\XX_0$ in $F$ regular and   with the property that 
 the special fibre  $X$  of $\XX$, the ramification locus of $F/F_0$ on $\XX$, 
 the ramification locus of $A$ on $\XX$ and the support of $\mu$ on $\XX$ are
a union of regular curves with normal crossings (\cite{A}, \cite{L}).
Let $X_0$ be the reduce special fibre of $\XX_0$ and $\{\eta_1, \cdots , \eta_m \}$ be the generic points of $X_0$.

Let  $\PP_0$ be a finite set of closed points of $X_0$ containing 
all the singular points of $X_0$ and at least one closed point from each irreducible component of $X_0$. 
Let $\UU_0$ be the  set of irreducible components of $X_0 \setminus \PP_0$.  
We fix the data $\mu \in F^*$,  $\XX_0, \PP_0$ and $\UU_0$ for  until (\ref{patching-su}).
Let $\BB_0$ be the set of branches at $\PP_0$. Since $X_0$ is a union of regular curves with normal crossings,  
$\BB_0$ is in bijection with pairs $(P, U)$ with $P \in \PP_0$, $U \in \UU_0$ and $P $ is in the closure of $U$. 

Let  $\eta \in X_0$ be  a generic point and   $P \in \overline{\{ \eta \} }$ a closed point. 
Then $\eta$ defines a discrete valuation $\nu_\eta$ on $F_{0P}$. Then the completion of 
$F_0$ at  the restriction of $\nu_\eta$  to $F_0$ is denoted by $F_{0\eta}$ and 
the completion of $F_{0P}$ at $\nu_\eta$ denoted by $F_{0P,\eta}$. 
The  closed point $P$ induces a discrete valuation $\nu_P$ on the residue field $\kappa(\eta)$ of 
$F_{0\eta}$ such that the completion $\kappa(\eta)_P$  of $\kappa(\eta)$ at $\nu_P$ is the residue field of 
$F_{0P, \eta}$.

Let $P \in X_0$ be a closed point and $A_{P}$ the local ring at $P$ on $\XX_0$. 
Since the normalization of $\XX_0$ in $F$ is regular, 
$d = \pi u$  or $d = u$ for some $\pi \in A_{P}$ a regular parameter and $u \in A_{P}$ a unit. 
Hence $B_P = A_P[\sqrt{d}]$ is the   integral closure of $A_P$ in $F_0$.  
Let $\delta \in A_P$ be such that $m_P = (\pi, \delta)$ be the maximal ideal of $A_P$.
If $d= \pi u $, then  $B_P$ is local and $(\sqrt{\pi u}, \delta)$ is the maximal ideal of $B_P$.  
Suppose $d=u$ a unit in $A_P$. If $u$ is not a square in the residue field $\kappa(P)$, then 
$B_P$ is local and the maximal ideal of $B_P$ is generated by $\pi$ and $\delta$.

 We begin with the following.   
 
 \begin{lemma}
 \label{correction-eta}
 Let $\eta$ be a generic point of $X_0$ and $S$ be a finite set of closed points of $\overline{\{ \eta \} }$.
 For every $P \in S$, let $\theta_{\eta, P} \in F_{0P, \eta}^*$ be  unit at $\eta$ which is  a reduced norm from $A\otimes F_{0P, \eta}$.
 Then there exists $\theta_\eta \in F_{0\eta}$ which is a reduced norm from $A \otimes F_{0\eta}$ such that 
 $\theta_\eta \theta^{-1}_{\eta, P} \in F_{0P, \eta}^{*n}$ for all $P \in S$. 
 \end{lemma}
 
 \begin{proof} Suppose $F_\eta = F \otimes F_{0\eta}/F_{0\eta}$ is a ramified field extension.
  Then    by (\ref{ramifiedh2}), 
 there exists an unramified algebra $A_0$ over $F_{0\eta}$ such that  
  $A \otimes _{F_0} F_{0\eta} \simeq A_0 \otimes_{F_0\eta} F_{\eta}$.  
  For $P \in S$, let $\overline{\theta}_{\eta, P} \in \kappa(\eta)^*_P$ be the image of 
  $\theta_{\eta, P} \in F_{0P, \eta}^*$.   
  We choose $\overline{\theta}_\eta \in \kappa(\eta)^*$ be close to $\overline{\theta}_{\eta, P}$ for all 
  $P \in S$.   Since $A_0$ is unramified  over $F_{0\eta}$, 
  its specialization   $B_0$  is a central simple algebra over $\kappa(\eta)$.
  Since $\kappa(\eta)$ is a global field of positive characteristic, 
  by Hasse-Maass-Schilling theorem $\overline{\theta}_\eta$ is a reduced norm from $B_0$.
  Let $\theta_\eta \in F_{0\eta}$ be a lift of $\overline{\theta}_\eta$.
   Since $F_{0\eta}$ is complete,  $\theta_\eta$ is a  reduced norm from $A_0\otimes_{F_0}F_{0\eta}$
   and hence a reduced norm from  $A\otimes_{F}F_{\eta}$.
   Since $\overline{\theta}_\eta$ is close to $\overline{\theta}_{\eta, P}$ for all $P \in S$
    and $n$ is coprime to char$(\kappa)$, 
    $\overline{\theta}_\eta \overline{\theta}_{\eta, P}^{-1}
   \in \kappa(\eta)_P^{*n}$ for all $P \in S$.
  Since $F_{0P, \eta}$ is complete with residue field $\kappa(\eta)_P$,  
   $\theta_\eta \theta_{\eta, P}^{-1} \in F_{0P, \eta}^{*n}$ for all $P \in S$.

Suppose that $F_\eta/F_{0\eta}$ is an unramified field extension. 
Then the residue field $\tilde{\kappa}( \eta)$  of $F_\eta$ 
is a quadratic extension of $\kappa(\eta)$.
Let $(L_\eta, \sigma_\eta)$ be the residue of  $A$ at $\eta$. 
Since the residue commutes with the corestriction, cores$_{\tilde{\kappa}(\eta) /\kappa(\eta)}(L_\eta, \sigma_\eta) = 0$.
Thus, by (\ref{galois}),   $L_\eta/\kappa(\eta)$ is a dihedral extension. 
Since $\theta_{\eta, P} $ is a reduced norm from $A\otimes F_{0P, \eta}$,  we have 
 $A \cdot (\theta_{\eta, P}) = 0 \in H^3(F\otimes F_{0P, \eta}, \mu_n)$. 
 Let $\overline{\theta}_{\eta, P}$ be the image  of   $\theta_{\eta, P}$  in the residue field $\kappa(\eta)_P$  of 
 $F_{P, \eta}$.   By  taking the residue of $A  \cdot (\theta_{\eta, P})$, we get that 
 $(L_\eta, \sigma_\eta, \overline{\theta}_{\eta, P}) = 0$ (cf. \cite[Proof of 4.7]{PPS}).
 Hence   $\overline{\theta}_{\eta, P}$   is a norm from the extension  $L_\eta \otimes_{\kappa(\eta)} \kappa(\eta)_P/ \tilde{\kappa}(\eta)  \otimes_{\kappa(\eta)} \kappa(\eta)_P$.
 Since $\kappa(\eta)$ is a  global field,  by (\ref{norms-approximation}), there exists $\overline{\theta}_\eta \in \kappa(\eta)^*$
 with $\overline{\theta}_\eta$  a norm from 
 $L_\eta/\tilde{\kappa}(\eta)$ and $\overline{\theta}_\eta \overline{\theta}_{\eta, P}^{-1} \in\kappa(\eta)_P^{*n}$.

Let $\theta_\eta \in F_{0\eta}$ be a  lift of $\overline{\theta}_\eta \in \kappa(\eta)$. 
Then $\theta_\eta\theta_{\eta, P}^{-1} \in F_{0P, \eta}^{*n}$. 
Since $\overline{\theta}_\eta$ is a norm from $L_\eta/\tilde{\kappa}(\eta)$,  
by (\ref{reduced-norms}), $\theta_\eta$ is a reduced norm from $A \otimes F_{0\eta}$. 

Suppose $F_\eta = F \otimes_{F_0} F_{0\eta}$ is not a field.  Then $F_\eta \simeq F_{0\eta} \times F_{0\eta}$
and $A \otimes_{F_0} F_{0\eta} \simeq A_1 \times A_1^{op}$, where $A_1^{op}$ is the opposite algebra.  Since 
$\theta_{\eta, P}  \in F_{0P, \eta}$ is a reduced norm from $A\otimes F_{0P, \eta}$, $\theta_{\eta, P}$ is a reduced norm from $A_1 \otimes F_{0P, \eta}$.
Then,  as above, we can find $\theta_\eta \in F_{0\eta}$ such that $\theta_\eta\theta_{\eta, P}^{-1} \in F_{0P, \eta}^{*n}$
and $\theta_\eta$ is a reduced norm from $A_1$. Then $\theta_\eta$ is a reduced norm from $A\otimes_{F_0} F_{0\eta}$. 
 \end{proof}
 
 \begin{lemma}
\label{choice-at-eta}
Suppose that  for every generic point $\eta$ of $X_0$ there exists 
 $c_\eta \in F_{0\eta}^*$ such that $\mu c_\eta $ is a reduced norm from  $A \otimes F_{0, \eta}$. 
 Then  for every  generic point $\eta$ of $X_0$, there exists 
 $a_\eta \in F_{0\eta}^*$ such that $\mu a_\eta$ is a reduced norm from  $A \otimes F_{0\eta}$
 with the following property:  if  $\eta_1$ and $\eta_2$ are two generic points of $X_0$ and 
  $P \in \overline{ \{\eta_1\}} \cap \overline{ \{\eta_2\}}$  with $F\otimes F_{0P}$ is a field,
  then  there exists $a_P \in F_{0P}^*$ such that 
$\mu a_P$   is a reduced norm from $A \otimes F_{0P}$
and $a_{\eta_i} a_P^{-1} \in F_{0P, \eta_i}^{*n}$ for $ = 1, 2$. 
\end{lemma}

\begin{proof}  Since the special fibre is a union of regular curves with normal crossings,  
 for a  generic point $\eta$ of $X_0$,  there exists $\pi_\eta \in F_0$   a parameter at $\eta$ such that 
 for every closed point  $P \in \overline{ \{\eta_1\}} \cap \overline{ \{\eta_2\}}$  for any two distinct 
 generic points $\eta_1$ and $\eta_2$ of $X_0$, the maximal ideal at $P$ is $(\pi_{\eta_1}, \pi_{\eta_2})$. 
 
 Suppose that  for every generic point $\eta$ of $X_0$ there exists 
 $c_\eta \in F_{0\eta}$ such that $\mu c_\eta $ is a reduced norm from  $A \otimes F_{0, \eta}$. 
 For every generic point $\eta$ of $X_0$, let $r_\eta = \nu_\eta(c_\eta)$.
 For every closed point  $P \in \overline{ \{\eta_1\}} \cap \overline{ \{\eta_2\}}$  with $F\otimes F_{0P}$ is a field,
 let $a_P  = \pi_{\eta_1}^{r_{\eta_1}} \pi_{\eta_2}^{r_{\eta_2}} \in F_{0P}^*$. 
 
Let $\eta$ be a generic point of $X_0$. Let $P \in \overline{ \{\eta \}} \cap \overline{ \{\eta'\}}$  for some generic point 
$\eta' \neq \eta$. Suppose that  $F\otimes F_{0P}$ is a field. By the choice of $\XX_0$, 
$F/F_0$ is unramified at $P$ except possibly at $\pi_{\eta}$ and $\pi_{\eta'}$. 
Since the maximal ideal at $P$ is $(\pi_{\eta}, \pi_{\eta'})$, by (\cite[Corollary 5.5]{PPS}), $F\otimes F_{0P, {\eta} }$ is a field. 
Since  $n$ is coprime to  char$(\kappa(P))$,   
 $c_\eta = u_P \pi_{\eta}^{r_{\eta}} \pi_{\eta'}^{s_P} (b_P)^n $ for some $s_P \in \Z$, $u_P \in \hat{A}_P$ a unit and 
 $b_P \in F_{0P, \eta}^*$ (cf. \ref{branch-elements}). 
Let $\theta_{\eta, P}  = u_P^{-1}\pi_{\eta'}^{r_{\eta'}-s_P}$.  
Since $F\otimes_{F_0}F_{0P}$ is a field,  $\theta_{\eta, P}$ is a reduced norm from $A\otimes_{F_0}F_{0P, \eta}$ (\ref{branch-norms}). 
Since $\theta_{\eta, P}$ is a unit at $\eta$,  by (\ref{correction-eta}), there exists $\theta_{\eta} \in F_{0\eta}$ 
which is a reduced norm from $A \otimes F_{0\eta}$ such that $\theta_\eta  \theta_{\eta, P}^{-1} \in F_{0P, \eta}^{*n}$. 

Let $a_\eta = c_\eta \theta_\eta$.  Since $\mu c_\eta$ and $\theta_\eta$ are reduced norms from $A \otimes F_{0\eta}$,
$\mu a_\eta$ is a reduced norm from $A \otimes F_{0\eta}$. 
Let $P \in \overline{ \{\eta \}} \cap \overline{ \{\eta'\}}$  for some generic point 
$\eta' \neq \eta$ with   $F\otimes F_{0P}$ is a field.
Then, by the choice of $a_\eta$, we have  $a_\eta = \pi_\eta^{r_\eta} \pi_{\eta'}^{r_{\eta'}}$ modulo $F_{0P, \eta}^{*n}$.
Hence $a_\eta a_P^{-1} \in F_{0P, \eta}^{*n}$. 
\end{proof}

\begin{lemma}
\label{choice-at-split-point}
  Let $\eta_1$  and $\eta_2$ be two distinct generic points of $X_0$.
Suppose  $P \in \overline{\{ \eta_1 \}} \cap \overline{\{ \eta_2 \}}$ is  a closed point with 
$F \otimes F_{0P}$ is not a field. 
Suppose there exist $a_{\eta_i} \in F_{\eta_i}^*$ with   
$\mu a_{\eta_i}$ is a reduced norm from  $A \otimes F_{0\eta_i}$. 
Then there exists $a_P \in F_{0P}^*$ such that 
$\mu a_P$   is a reduced norm from $A \otimes F_{0P}$
and $a_{\eta_i} a_P^{-1} \in F_{0P, \eta_i}^{*n}$ for $ i = 1, 2$. 
\end{lemma}

\begin{proof} Since $F\otimes F_{0P}$ is not a field and cores$_{F/F_0}(A) = 0$, 
$F\otimes F_{0P} \simeq F_{0P} \times F_{0P}$
and $A \otimes F_{0P} \simeq A_1 \times A_1^{op}$ for some central simple algebra 
$A_1$ over $F_{0P}$. 
Write $\mu = (\mu_1, \mu_2)$. 
Since $a_{\eta_i}\mu$ is a reduced norm  from $A \otimes F_{\eta_i} $and  $a_{\eta_i} \in F_{0\eta_i}^*$, 
$a_{\eta_i}\mu_1$  and $\mu_1\mu_2^{-1}$  are reduced norms from $A_1 \otimes F_{0P,  \eta_i}$. 
Since by the choice of $\XX_0$, the support of $\mu$ on $\XX$ and the ramification locus of $A$ on
$\XX$  is a union of regular curves with normal crossings,  by (\ref{2dim-local-nrd}),
$\mu_1\mu_2^{-1}$ is a reduced norm from $A_1\otimes F_{0P}$.

The generic points $\eta_1$ and $\eta_2$ give discrete valuations $\nu_1$ and $\nu_2$ on $F_{0P}$
with completions $F_{0\eta_1, P}$ and $F_{0\eta_2, P}$.
Let $z_i  \in A_1 \otimes F_{0P, \eta_i}$ with reduced norm $a_{\eta_i}\mu_1$.
Let $z \in A_1 \otimes F_{0P}$ be close to $z_i$ for $i = 1, 2$. 
Let $a_P = \mu_1^{-1}Nrd(z) \in F_{0P}$.
Then $\mu_1 a_P $ is a reduced norm from $A_1 \otimes  F_{0P}$.
Since $z$ is close to $z_i$ and Nrd$(z_i) = a_{\eta_i}\mu_1$,
Nrd$(z)$ is close to $a_{\eta_i} \mu_1$.  Hence $a_P$ is close to $a_{\eta_i}$.
Therefore $a_{\eta_i}a_P^{-1} \in F_{0P, \eta_i}^{*n}$.
Since $\mu_1\mu_2^{-1}$ is a reduced norm and $a_P \mu_1$ is a reduced norm, 
$a_P\mu_2$ is a reduced norm. In particular $a_P \mu$ is a reduced norm. 
\end{proof}

\begin{lemma}
\label{choice-at-curve-point}
   Let $\eta$  be a  generic point of $X_0$ and  $P \in \overline{\{ \eta\}} $   a closed point. 
Suppose there exist $a_{\eta} \in F_{\eta}$ with   
$\mu a_{\eta}$ is a reduced norm from  $A \otimes F_{0\eta}$. 
Then there exists $a_P \in F_{0P}^*$ such that 
$\mu a_P$   is a reduced norm from $A \otimes F_{0P}$
and $a_{\eta} a_P^{-1} \in F_{0P, \eta}^{*n}$.  
\end{lemma}

\begin{proof}  Suppose that $F\otimes F_{0P}$ is a field.
Then,  by the choice of $\XX_0$ and by (\ref{2dim-local}), 
there exists $a_P \in F_{0P}$ such that 
$\mu a_P$   is a reduced norm from $A \otimes F_{0P}$
and $a_{\eta} a_P^{-1} \in F_{0U, P}^n$. 

Suppose $F\otimes F_{0P}$ is a not  field. 
Then, we get the required $a_P$ as in the proof of (\ref{choice-at-split-point}).
\end{proof}

\begin{theorem}
\label{patching-su}
Let $K$ be a  local field   with the residue field $\kappa$ and valuation ring $T$.
Let $F_0$ be the function field of a smooth projective   curve over $K$ and  $F/F_0$  a separable quadratic extension.
Let $A$ be a central simple algebra over $F$  of degree $n$ with an involution $\sigma$ of second kind and 
$F^{\sigma} = F_0$. Suppose that $2n$ is coprime to char$(\kappa)$.   
Let  $\XX_0 \to Spec(T)$ be a proper normal model  of $F_0$ with special fibre $X_0$.
Let $\PP_0$ be a finite set of closed points of  $X_0$ containing all the singular points of   $X_0$
and $\UU_0$ the set of irreducible components of $X_0 \setminus \PP_0$. 
 Then the canonical map   $$H^1(F_0, SU(A, \sigma)  \to \prod_{U \in \UU_0}   H^1(F_{0U}, SU(A, \sigma)) \times \prod_{P\in \PP_0} H^1(F_{0P}, SU(A, \sigma))$$
 has trivial kernel. 
 \end{theorem} 

\begin{proof} 
Let $\xi \in H^1(F_0, SU(A, \sigma))$. Suppose that $\xi $ maps to 0 in $H^1(F_{0x}, SU(A, \sigma))$ for all $x \in \UU_0 \cup \PP_0$.
Since $U(A, \sigma)$ is rational  and connected    (\cite[p. 195, Lemma 1]{Merkurjev}), by (\cite[Theorem 3.7]{HHK1}), $\xi$ maps to 0 in $H^1(F_0, U(A, \sigma))$. 
Hence from the  exact sequence ($\star$ of \S \ref{lgp-patching}), there exists $\lambda \in F^{*1}$  such that 
$\lambda$   maps to $\xi$ in $H^1(F_0, SU(A, \sigma)$.
  Let $\mu \in F^*$ be such that $\lambda = \mu^{-1} \sigma(\mu)$.  
  Since $\xi$ maps to 0 in $H^1(F_{0U}, SU(A, \sigma))$, there exists $c_U \in F_{0U}$ such that 
  $c_{U} \mu $ is a reduced norm from $A\otimes_{F_0}F_{0U}$ (cf. (\cite[p. 202]{KMRT})). 
  
Then, there exists a sequence of blow-ups $\XX_0' \to \XX_0$ such that  $\XX_0'$ is regular, 
the integral closure  $\XX'$  of $\XX_0'$ in $F$ is regular and the union of the  special fibre of $\XX'$,
the ramification locus of $A$ on $\XX'$ and the support of $\mu$ on $\XX'$ is a union of regular curves with normal crossings.
Let $\PP_0'$ be a finite set of closed points of $\XX_0'$ containing all the singular points of the special fibre $X_0'$ of $\XX_0'$
and  at least one closed point lying over   points of $\PP_0$. Let $\UU_0'$ be the set of components  of $X_0' \setminus \PP_0'$.
Then $\xi $ maps to 0 in $ H^1(F_{0x'}, SU(A, \sigma))$ for all $x' \in  \PP_0' \cup \UU_0'$ (\cite[\S 5]{HHK3}).
Thus, replacing $\XX_0$ by $\XX_0'$, we assume that 
the integral closure  $\XX$  of $\XX_0$ in $F$ is regular and the union of the  special fibre of $\XX$,
the ramification locus of $A$ on $\XX$ and the support of $\mu$ on $\XX$ is a union of regular curves with normal crossings.

Let $\eta$ be a generic point of $X_0$. Then $\eta \in U_\eta$ for some $U_\eta \in \UU$.
Let $c_\eta = c_{U_\eta}$. Since $F_{0U_\eta} \subset F_{0\eta}$, $c_\eta \in F_{0\eta}^*$
and $c_\eta \mu $ is a reduced norm from $A\otimes_{F_0}F_{0\eta}$. 
Let   $a_\eta \in F_{0\eta}$ be as in (\ref{choice-at-eta}). 
  Then,  by Artin's approximation (\cite[Theorem 1.10]{Artin},   as in the proof of    (\cite[Lemma 7.2]{PPS}), 
  there exists a nonempty open subset $V_\eta$ of $U_\eta$ such that 
$a_\eta \in F_{0V_\eta}$ (\cite[Lemma 3.2.1]{HHK4}) and $a_\eta \mu $ is a reduced norm from 
$A\otimes_{F_0}F_{0V_\eta}$.   Let  $a_{V_\eta} = a_\eta^h \in F_{0V_\eta}$.
Let $\UU'$ be the set of these $V_\eta$'s.
Let $\PP_0'$ be the complement of the union of $V_\eta$'s in $X_0$.
Then  $\UU' $  is the set of components of $X_0 \setminus \PP_0'$.

Let $P \in \PP'_0$. Suppose that $P \in \overline{\{\eta\}}  \cap \overline{\{\eta'\}}$ for two distinct 
generic points $\eta$ and $\eta'$ of $X_0$.  Then $P \in \PP_0$. 
If $F\otimes F_{0P}$ is a field, then let $a_P \in F_{0P}$ be as in (\ref{choice-at-eta}).
If $F\otimes F_{0P}$ is not a field,  let $a_P \in F_{0P}$ be as in (\ref{choice-at-split-point}).
Suppose $P \in \overline{\{\eta\}}$ for some  generic point $\eta$ of $X_0$
and   $P \not\in \overline{\{\eta'\}}$ for all  generic points $\eta'$ of $X_0$ not equal to $\eta$. 
Let $a_P \in F_{0P}$ be as in (\ref{choice-at-curve-point}).

Let $(V, P)$ be a branch. Then  $P \in \overline{V}$.  By the choice of $a_P$ and $a_V$, we have 
$a_V a_P^{-1} = b_{V,P}^n$ for some $b_{V, P} \in F_{0V, P}^*$.
By  (\cite[Corollary 3.4]{HHK3}), for every $x \in \UU'_0 \cup \PP'_0$, there exists $b_x \in F_{0x}^*$ such that 
$b_{V, P} = b_Vb_P$ for all branches $(V, P)$.

For $V \in \UU_0$, let  $a_V' =  a_V b_V^{-n}$ 
and for $P \in \PP'_0$, let $a_P' = a_P b_P^n$. 
Then, we have $a_V'  = a'_P \in F_{0V, P}$ for all branches $(V, P)$.
Hence there exists $a' \in F_0$ such that $a' = a'_x \in F_{0x}$ for all $x \in \UU'_0 \cup \PP'_0$ (\cite[\S 3]{HHK3}).
Since $\mu a_x $ is a reduced norm from $A \otimes F_{0x}$ for all $x \in \UU'_0 \cup \PP'_0$
and $n$ is the degree of $A$, $\mu a'_x$ is a reduced norm from $A \otimes F_{0x}$ for all $x \in \UU'_0 \cup \PP'_0$. 
Thus, by (\cite[Corollary 11.2]{PPS} and  \cite[Proposition 8.2]{HHK3}), $\mu a'$ is a reduced norm from $A$. 
Since $\lambda = (\mu a')^{-1} \sigma(\mu a')$, $\lambda$ is in the   image of 
$U(A, \sigma)(F_0) \to F^{*1}$ and hence $\xi$ is trivial. 
\end{proof}

The following is immediate from (\ref{patching-su}) and (\cite[Corollary 5.9]{HHK3})

\begin{cor}
\label{pointsha-su}
 Let $K$ be a local field with residue field $\kappa$.
Let $F_0$ be a function field of  a curve over $K$.
Let $F/F_0$ be a quadratic extension and $A$ a central simple algebra over $F$ of index $n$
 with an $F/F_0$- involution $\sigma$. 
 Suppose that $2n$ is coprime to char$(\kappa)$. 
 Then the canonical map   $$H^1(F_0, SU(A, \sigma))  \to \prod_{x \in X_0}   H^1(F_{0x}, SU(A, \sigma))  $$
 has trivial kernel. 
 \end{cor} 

\section{Local global principle for special unitary groups - discrete valuations}
\label{lgp-dvr}

\begin{theorem}
\label{lgp-unitary}
Let $K$ be a local field with residue field $\kappa$.
Let $F_0$ be a function field of  a curve over $K$.
Let $F/F_0$ be a quadratic extension and $A$ a central simple algebra over $F$ of index $n$
 with an $F/F_0$- involution $\sigma$. 
 Suppose that $2n$ is coprime to char$(\kappa)$. 
 Then the canonical map   $$H^1(F_0, SU(A, \sigma) ) \to \prod_{\nu \in \Omega_{F_0}} H^1(F_{0\nu}, SU(A, \sigma))$$
 has trivial kernel. 
\end{theorem} 

\begin{proof} Let $\xi \in H^1(F_0, SU(A, \sigma))$. Suppose that $\xi$ maps to 0 in $H^1(F_{0\nu}, SU(A,\sigma))$ for all $\nu \in \Omega_{F_0}$. 
By (\ref{uas}),   the image of $\xi$ in $H^1(F_0, U(A, \sigma))$ is zero.
Hence from the  exact sequence ($\star$) of \S \ref{lgp-patching}, there exists $\lambda \in F^{*1}$  such that 
$\lambda$   maps to $\xi$ in $H^1(F_0, SU(A, \sigma)$.
Write $\lambda = \mu^{-1} \sigma(\mu)$ for some $\mu \in F^*$. 

 Let $d \in F_0^*$ be such that 
$F = F_0(\sqrt{d})$. 
There exists a  regular proper model $\XX_0$  of $F_0$ such that the special fibre and the support of 
$d$ is a union of regular curves with normal crossings. Further the integral closure $\XX$  of $\XX_0$ in $F$ has the property:
$\XX$ is regular, 
the special fibre of $\XX$, the ramification locus  of $(A, \sigma)$,  the support of $d$, $\mu$ and $\lambda$
 is a union of regular curves  with normal crossings (cf. \cite{Wu}). Let $X_0$ be the special fibre of $\XX_0$. 

Let $x \in X_0$ be codimension zero point. Then $x$ gives a discrete valuation  $\nu_x$ on $F_0$ and $F_{0x} = F_{0\nu_x}$.
Hence $\xi$ maps to zero in  $H^1(F_{0x}, SU(A,\sigma))$.

Let $P \in X_0$ be a closed point.  Let $A_P$ be the local ring at $P$ and $B_P$ the integral closure of 
$A_P$ in $F$.  Since $B_P$ is regular, there is at most one  irreducible 
curve  of $\XX$  in the support of $d$ which passes  through $P$. 
Further there are at most two curves passing through $P$ which are in the union of special fibre of $\XX'$, 
the support of $\mu$ and ramification locus of $A$.  Let $x$ be one such curve and $\nu_x$ the discrete valuation of $F_0$ given by $x$. 
Then $F_{0 \nu_x} \subset F_{0P, \nu_x}$ and hence  $\xi$ maps to 0 in $H^1(F_{0P, \nu_x}, SU(A, \sigma))$.

Since $\lambda$ maps to $\xi$, there exists $\theta_x \in F_{0P, \nu_x}$ such that 
$\mu \theta_x \in Nrd(A\otimes F_{0P, \nu_x})$. Hence,
by (\ref{2dim-local}), there exists $\theta_P \in F_{0P}$ such that $\mu\theta_P \in Nrd(A\otimes F_P)$.
In particular $\xi \otimes F_P$ is trivial.  Hence, by (\ref{pointsha-su}), $\xi$ is trivial.
\end{proof}

\section{Conjectures 1 and 2 for classical groups} 
\label{conjectures}

In this section we prove the validity of Conjecture 1 and Conjecture 2 for all groups of classical type in the good 
characteristic case. In fact we prove local global principles for  function fields of curves over any local field.

  \begin{theorem}
 \label{conjecture1}
  Let $K$ be a  local  field with residue field $\kappa$ and  $F$  the  function field of  a curve over $K$.
  Let $G$ be a connected linear algebraic group over $F$ of classical type ($D_4$ nontrialitarian) with char$(\kappa)$ good for $G$. 
  Let $Z$ be a projective homogeneous space under $G$ over $F$.
  If $Z(F_\nu) \neq \emptyset$ for all divisorial discrete valuations of $F$, then $Z(F_\nu) \neq \emptyset$. 
     \end{theorem}   
   
   \begin{proof}  Let $G^{ss}$ be the semisimplification of $G/rad(G)$.
    Since $G$ is of classical type, there exists     a central 
    isogeny $G_1\times \cdots  \times G_r \to G^{ss}$ with each $G_i$  an almost simple simply connected group of 
   the classical type  ($D_4$ nontrialitarian) with char$(\kappa)$  good. 
   It is well know that using the results of (\cite[Theorem 2.20]{BT}) and (\cite[Proposition 6.10]{MPW2}), 
   one reduces to the case $r = 1$ (cf.  proof of \cite[Corollary 4.6]{Wu}). 
   
   Let $G$ be a connected linear algebraic group  with an  isogeny $G' \to G^{ss}$ 
   for some almost simple connected group  $G'$ of 
   classical type  ($D_4$ nontrialitarian).  
    If $G'$ is of type $^1\!\!A_n$, then the  result follows   from (\ref{proj-pgl}).  
    If $G'$ is of type $^2\!\!A_n$, then the  result    follows from (\ref{isotropic1}). 
    If $G'$ is of type $B_n$, $C_n$ or $D_n$, then the result follows from  (\cite{Wu}). 
   \end{proof}
  
  \begin{theorem} 
   \label{conjecture2}
 Let $K$ be a local field with residue field $\kappa$ and  $F$  the  function field of  a curve over $K$.
  Let $G$ be a  semisimple simply connected  linear algebraic group over $F$ with   char$(\kappa)$  is   good for $G$. 
  Suppose $G$ is of the classical type ($D_4$ nontrialitarian). 
  Let $Z$ be a  principal   homogenous space under $G$ over $F$.
  If $Z(F_\nu) \neq \emptyset$ for all divisorial discrete valuations of $F$, then $Z(F_\nu) \neq \emptyset$. 
   then Conjecture 2 holds for $G$.
   \end{theorem}

   \begin{proof} For  $G$   of type $B_n$, $C_n$ or $D_n$  ($D_4$ nontrialitarian), 
   this result  is   proved in (\cite{Hu} and \cite{preeti}). 
   
   Suppose  $G$ is of type $^1\!\!A_n$. Then 
   $G \simeq SL(A)$ for some central simple algebra $A$ over $F$ and the principal homogeneous spaces 
   under $G$ are classified by $H^1(F, G) \simeq F^*/Nrd(A)$.
   Since char$(\kappa)$ is good for $G$, the degree of $A$ is coprime to char$(\kappa)$. 
   Hence the result   follows from (\cite[Corollary 11.2]{PPS}). 
   
   Suppose $G$ is of type $^2\!\!A_n$. 
   Then  there exists a separable quadratic extension $F/F_0$ and central simple algebra over $F$ with 
   an $F/F_0$-involution $\sigma$ such that $G \simeq SU(A, \sigma)$.
   Since   char$(\kappa)$ is   good for $G$,  2(deg$(A)$) is coprime to char$(\kappa)$. 
    Hence the result      follows from (\ref{lgp-unitary}). 
    \end{proof}

\providecommand{\bysame}{\leavevmode\hbox to3em{\hrulefill}\thinspace}

\end{document}